\documentclass {siamltex}
\usepackage{color}
\usepackage{graphicx}
\usepackage{epsfig}
\usepackage{latexsym}
\usepackage{array}
\usepackage{amssymb}
\usepackage{amsmath}
\usepackage{amsbsy}
\usepackage{stmaryrd}

\newcommand{\jump}[1]{\llbracket#1\rrbracket}

\newtheorem{remark}{\it Remark\/}
\usepackage{eepic}
\title{Stabilised finite element methods for nonsymmetric,
  noncoercive and ill-posed problems. Part I: elliptic equations}

\author{Erik Burman\thanks{Department of Mathematics, 
University College London, London, 
UK--WC1E  6BT, 
United Kingdom; ({\tt e.burman@ucl.ac.uk})}
}
 
\begin{document}

\maketitle

\begin{abstract}
In this paper we propose a new method to stabilise nonsymmetric indefinite
problems. The idea is to solve a forward and an adjoint problem
simultaneously using a suitable stabilised finite element method. Both
stabilisation of the
element residual and of the jumps of certain derivatives of the discrete
solution over element faces may be used. Under the assumption
of well posedness of the partial differential equation and its
associated adjoint problem we
prove optimal error estimates in $H^1$ and $L^2$ norms in an abstract
framework. Some examples of
problems that are neither symmetric nor coercive, but that enter the
abstract framework are given. First we treat indefinite convection-diffusion equations, with
non-solenoidal transport velocity and either pure Dirichlet
conditions or pure Neumann conditions and then a Cauchy
problem for the Helmholtz operator. Some numerical
illustrations are given.
\end{abstract}


\newcommand{\cut}{c}
\def\IR{\mathbb R}
\def\Ext{\mbox{\textsf{E}}}

\section{Introduction}
The computation of indefinite elliptic problems often involves certain
conditions on the mesh size $h$ for the system to be well-posed and
for the derivation of error estimates. The first results on this
problem are due to Schatz \cite{Schatz74}. 
The conditions on the mesh parameter can be avoided if a stabilised
finite element method is used. 
Such methods have been proposed by Bramble et al. \cite{BLP98,Ku07} or
more recently the continuous interior penalty (CIP) method for the Helmholtz equation suggested by
Wu et al. \cite{Wu11,ZhuBuWu12}. The method proposed herein has some common
features with both these methods, but appears to have a wider field of
applicability. We may treat not only symmetric indefinite problems
such as the (real valued) Helmholtz equation, but also nonsymmetric indefinite
problems such as convection--diffusion problems with non-solenoidal
convection velocity or the Cauchy problem. The latter problem is known
to be ill-posed in general \cite{Belg07} and will mainly be explored
numerically herein. 
For all these cases we show that if the
primal and adjoint problems admit a unique
solution with sufficient smoothness the proposed algorithm converges
with optimal order. The case of hyperbolic problems is treated in the
companion paper \cite{part2}.

The idea of this work is to assume ill-posedness of the discrete form
of the PDE and regularise it in the form of an optimisation problem under
constraints. Indeed we seek to minimise the size of the stabilisation
operator under the constraint of the discrete variational form.
The regularisation terms are then chosen from well known
stabilised methods respecting certain design criteria
given in an abstract analysis. This leads to an extended method where
simultaneously both a primal and a dual problem are solved, but where
the exact solution of the dual problem is always trivial. The aim is
to obtain a method where possible discrete non-uniqueness is
alleviated by discrete regularisation, with a non-consistency that can be
controlled so that optimal convergence for smooth solutions is obtained. The
method is also a good candidate for cases where the solution to the continuous problem 
is non-unique, but that is beyond the scope of the present paper.

In spite of the lack of coercivity for the
physical problem, the discrete problem has partial coercivity on the
stabilisation operator. A consequence of this is that, depending on
the kernel of the stabilisation operator 
a unique discrete solution 
may often be shown to exist independently of the underlying partial
differential equation. This can be helpful when exploring
ill-posed problems numerically or when measurement errors in the data,
may lead to an ill-posed problem, although the true problem is well
posed. 

An outline of the paper is as follows, in Section \ref{sec:abstract} we
propose an abstract method and prove that the method will have optimal
convergence under certain assumptions on the bilinear form. Then in
Section \ref{sec:stabilisation} we discuss stabilised methods that satisfy
the assumptions of the abstract theory with particular focus on the
Galerkin least squares method (GLS) and the continuous interior penalty
method (CIP). Three example of applications are given in Section
\ref{sec:applications}, two different noncoercive transport problems in
compressible flow fields and one elliptic Cauchy problem. Finally in
Section \ref{sec:numerical} the 
accuracy and robustness of the proposed method is shown by some
computations of solutions to the problems discussed in Section \ref{sec:applications}.
In particular we study the performance of the approach for some different Cauchy
problems of varying difficulty.

\section{Abstract formulation}\label{sec:abstract}
Let $\Omega$ be a polygonal/polyhedral subset of $\mathbb{R}^d$, $d=2,3$. The
boundary of $\Omega$ will be denoted $\partial \Omega$ and its normal $n$.
For simplicity we will reduce the scope to second order elliptic
problems, but the methodology can readily be extended to indefinite
elliptic problems of any order, providing the operator has a smoothing
property. We will also describe the method mainly in the
two-dimensional case, only mentioning the dimension when the two and three
dimensional cases differ.

We let $V,W$ denote two subspaces of $H^1(\Omega)$.
The abstract weak formulation of the continuous problem takes
the form: find $u \in V$ such that 
\begin{equation}\label{forward}
a(u,w) = (f,w), \quad \forall w \in W
\end{equation}
with formal adjoint: find $z \in W$ such that
\begin{equation}\label{adjoint}
a(v,z) = (g, v) , \quad \forall v \in V.
\end{equation}
The bilinear form $a(\cdot,\cdot):V\times W
\rightarrow \mathbb{R}$ is assumed to be elliptic, but neither
symmetric nor coercive. We denote the forward problem on strong form
$\mathcal{L} u = f$ and the adjoint problem on strong form
$\mathcal{L}^* z = g$. Suitable boundary conditions are integrated
either in the spaces $V,W$ or in the linear form.

We assume that both these problems are well posed and that the
geometry and data are such that the
smoothing property holds
\begin{equation}\label{smooth}
|u|_{H^2(\Omega)} \leq c_{a,\Omega} \|f\|,\quad |z|_{H^2(\Omega)}
\leq c_{a,\Omega}\|g\|.
\end{equation}
We will frequently use the notation $a \lesssim b$  for $a \leq C b$ with $C$ a
constant depending only on the mesh
geometry and physical parameters giving an order one contribution. 
We will also use $a \sim b$ for $a \lesssim b$ and $b \lesssim a$. Indexed constants $c_{xy}$ will depend on
the variables $xy$, but can differ  at each occurrence.

The $L^2$-scalar product over some $X \subset \mathbb{R}^d$ is denoted $(\cdot,\cdot)_X$ and the associated
norm $\|\cdot\|_X$, the subscript is dropped whenever $X = \Omega$. We
will also use $\left<\cdot,\cdot\right>_Y$ to denote the $L^2$-scalar
product over $Y \subset \mathbb{R}^{d-1}$ and $(\cdot,\cdot)_h$ the
element-wise $L^2$-norm with the associated broken norm $\|\cdot\|_h$.
\textcolor{black}{
\begin{remark}
The above regularity assumptions are necessary to ensure optimal
convergence for piecewise affine approximation spaces. If 
polynomial approximation of order $k$ is used we additionally need $u
\in H^{k+1}(\Omega)$. If on the other hand the solution is less
regular the convergence order is reduced in the standard way and in
some cases the mesh constraints for well-posedness will be stronger. More
precisely if $u\in H^s(\Omega)$ and $z \in H^t(\Omega)$ with $s,t \in
(1,2)$ the analysis below leads to estimates on the form
\[
\|u - u_h\| \lesssim h^{s+t-2}.
\]
\end{remark}
}
\subsection{Finite element discretisation}
Let $\{\mathcal{T}_h\}_h$ denote a family of quasi uniform, shape
regular triangulations $\mathcal{T}_h:=\{K\}$, indexed by the maximum
triangle radius $h:= \max_{K \in \mathcal{T}_h} h_K$, $h_K :=
\mbox{diam}(K)$. The set of faces of the
triangulation will be denoted by $\mathcal{F}$ and
$\mathcal{F}_{int}$ denotes the subset of interior faces.
Now let $X_h^k$ denote the finite element space of continuous, piecewise 
polynomial functions on
$\mathcal{T}_h$, 
$$
X_h^k := \{v_h \in H^1(\Omega): v_h\vert_{K} \in \mathbb{P}_k(K),\quad \forall K
\in \mathcal{T}_h\}.
$$
Here $\mathbb{P}_k(K)$ denotes the space of polynomials of degree
less than or equal to $k$ on a triangle $K$.

We let $\pi_L$ denote the standard $L^2$-projection onto $X_h^k$ and
$i_h:C^0(\bar \Omega) \mapsto X_h^k$ the standard Lagrange interpolant. Recall that
for any function $u \in (V \cup W)\cap H^{k+1}(\Omega)$ there holds
\begin{equation}\label{approx}
\|u - i_h u\| + h \|\nabla(u - i_h u)\|+ h^2 \|D^2(u - i_h  u)\|_h  \leq c_i h^{k+1} |u|_{H^{k+1}(\Omega)}
\end{equation}
and under our assumptions on the mesh, similarly for $\pi_L$. 
We propose the following finite element method for the approximation
of \eqref{forward}, find $u_h,z_h \in V_h \times W_h$ such that
\begin{equation}\label{FEM}
\begin{array}{rcl}
a_h(u_h,w_h) + s_a(z_h,w_h) &=&(f,w_h) \\[3mm]
a_h(v_h,z_h) - s_p(u_h,v_h) &=& -s_p(u,w_h),
\end{array}
\end{equation}
for all $v_h,w_h \in V_h \times W_h$.
\textcolor{black}{ Note the appearance of
$s_p(u,w_h)$ in the right hand side of the second equation of
\eqref{FEM}. This means that only stabilisations for which
$s_p(u,w_h)$ can be expressed using known data may be used. For
residual based stabilisations this typically is the case, but also for
so-called observers that stabilise the computation using measured
data.
We will always assume that $u$ is
sufficiently regular so that $s_p(u,\cdot)$  is well defined, i.e. the
stabilisation is strongly consistent. The
analysis using weak consistency of the stabilisation is a
straightforward modification.
}
The bilinear form
$a_h(\cdot,\cdot)$ is a discrete realisation of $a(\cdot,\cdot)$,
typically modified to account for the effect of boundary conditions,
since in general $V_h \not \in V$ and $W_h \not \in W$. The penalty
operators
$s_a(\cdot,\cdot)$ and $s_p(\cdot,\cdot)$ are symmetric stabilisation
operators and associated
with the adjoint and the primal equation respectively. 

The rationale
of the
formulation may be explained in an optimisation
framework. Assume that we want to solve  the problem, find $u_h \in
V_h$ such that 
\[
a_h(u_h,w_h) = (f,w_h),\quad \forall w_h \in W_h,
\]
but that the system matrix corresponding to $a_h(u_h,w_h)$ has zero
eigenvalues. The discrete system is ill-posed. This often reflects
some poor stability properties of the underlying continuous problem.
The idea is to introduce a selection criterion for the solution, in
order to ensure discrete uniqueness, 
measured by some operator $s_p(u_h,v_h)$. This can include both
stabilisation (regularisation) terms and the fitting of the computed solution to
measurements. The formulation then writes, find $u_h,z_h
\in V_h \times W_h$ stationary point of the Lagrangian
\begin{equation}\label{Lagrange}
\L(u_h,z_h):= \frac12 s_p(u_h-u,u_h-u) - \frac12  s_a(z_h,z_h) -
a_h(u_h,z_h) + (f,z_h). 
\end{equation}
The saddle point structure of the Lagrangian has been enhanced by
the addition of the regularising term $-\tfrac12 s_a(z_h,z_h) $.
We may readily verify that
\[
\frac{\partial \L}{\partial u_h}(v_h) = s_p(u_h-u,v_h) - a_h(v_h,z_h) 
\]
and
\[
\frac{\partial \L}{\partial z_h}(w_h) = -a_h(u_h,w_h) - s_a(z_h,w_h)+(f,w_h).
\]
It follows that \eqref{FEM} corresponds to the optimality conditions of
\eqref{Lagrange}.

 

Observe that the second equation of \eqref{FEM} is a finite
element discretisation of the dual problem \eqref{adjoint} with data
$g=0$. Hence the solution to approximate is $z=0$. The discrete
function $z_h$ will most likely not be zero, since it is perturbed by the
stabilisation operator acting on the solution $u_h$, which in general
does not coincide with the stabilisation acting on $u$.

We will
assume that the following strong consistency properties hold. If $u$
is the solution of \eqref{forward} then
\begin{equation}\label{consist1}
a_h(u,\varphi) =(\mathcal{L }u,\varphi) \mbox{ for all }\varphi \in W + W_h
\end{equation}
and if $z$ is the solution of \eqref{adjoint} then
\begin{equation}\label{consist2}
a_h(\phi,z) = (\phi,\mathcal{L}^* z) \mbox{ for all }\phi \in V+V_h.
\end{equation}
As a consequence the following Galerkin orthogonalities hold
\begin{equation}\label{galortho}
a_h(u - u_h,v_h) = s_a(z_h,v_h) \mbox{ and } a_h(w_h, z - z_h) =
s_p(u -u_h ,w_h).
\end{equation}
The bilinear forms
$s_a(\cdot,\cdot)$, $s_p(\cdot,\cdot)$ are  symmetric, positive semi-definite, weakly consistent, stabilisation
operators. The semi-norms on $V_h$ and $W_h$ associated to the stabilisation is defined by
\[
|x_h|_{S_y} := s_y(x_h,x_h)^{\frac12}, \quad y = a,p.
\]
The rationale of
this formulation is that the following partial coercivity is obtained
by taking $w_h = z_h$ and $v_h=u_h$,
\begin{equation}\label{partial_coerc}
|z_h|_{S_a}^2 + |u_h|_{S_p}^2= (f, z_h) - s_p(u,u_h).
\end{equation}

We assume that there are interpolation operators $\pi_V:V 
\rightarrow V_h $ and $\pi_W : W \rightarrow W_h$, satisfying
\eqref{approx} and also that the following continuity relations hold for all $v,w,y
\in H^{2}(\Omega)$ and for all $v_h, x_h \in W_h$
\begin{equation}\label{cont1}
a_h(v - \pi_V v, x_h) \leq \|v - \pi_V v\|_+ (c_a  |x_h|_{S_a}+ \epsilon(h)\|x_h\|)
\end{equation}
and
\begin{equation}\label{cont2}
a_h(v - v_h,y - \pi_W y) \leq \|y - \pi_W y\|_* (c_a \|v- \pi_V
v\|_{\mathcal{L}} +
c_a |v_h - \pi_V v|_{S_p}+ \epsilon( h)\|v-v_h\|).
\end{equation}
We have introduced the notation $\|\cdot\|_+$, $\|\cdot\|_*$  and
$\|\cdot\|_{\mathcal{L}}$, for semi-norms to be
defined. These norms, and those induced by the stabilisation operators, will be assumed to satisfy the approximation estimates
\begin{equation}\label{approx1}
\|v - \pi_V v\|_{\mathcal{L}} + \|v - \pi_V v\|_+ + |v- \pi_V v|_{S_p}  \leq c_{a,\gamma} h^k
|v|_{H^{k+1}(\Omega)}, \quad \forall v \in V \cap H^{k+1}(\Omega),
\end{equation} 
\begin{equation}\label{approx2}
 \|w - \pi_W w\|_* + |w -\pi_W w|_{S_a} \leq  c_{a,\gamma} h^k |w|_{H^{k+1}(\Omega)} \quad \forall w \in W \cap H^{k+1}(\Omega),
\end{equation}
and the additional upper bounds
\begin{equation}\label{H2}
|\pi_W w|_{S_a} \leq  c_{a,\gamma} h|w|_{H^{2}(\Omega)} ,\,\forall w
\in W \cap H^{2}(\Omega), \,|\pi_V v|_{S_p} \leq  c_{a,\gamma} h|v|_{H^{2}(\Omega)} \, \forall v
\in V \cap H^{2}(\Omega).
\end{equation}
Here $c_{a,\gamma}$ depends on the form $a(\cdot,\cdot)$ and a
stabilisation parameter $\gamma$. 
\subsection{Convergence analysis for the abstract method}
We first prove that the stabilisation semi-norm of the discrete error is bounded by one term that converges to
zero at an optimal rate and another non-essential perturbation that
can be made small.
\begin{lemma}\label{stab_conv}
Assume that that
   the solution of \eqref{forward} is smooth, that the forms of
   \eqref{FEM} and the operators $\pi_V$, $\pi_W$ 
    are such that \eqref{galortho}, \eqref{cont1} and \eqref{approx1}
    are satisfied. Then for and $u_h,z_h$ solution of \eqref{FEM} there holds
\[
|\pi_Vu - u_h|_{S_p}  + |\pi_W z - z_h|_{S_a}\lesssim c_{a,\gamma,\epsilon}
h^k
|u|_{H^{k+1}(\Omega)} + \epsilon(h)  \|z_h\|,
\]
where $c_{a,\gamma,\epsilon} = c_{a,\gamma}  (1 + c_a^2)^{\tfrac12} $, with $c_a$ and $c_{a,\gamma}$ defined by
\eqref{cont1} and \eqref{approx1} respectively.
Similarly, if $s_p(u,w+w_h)=0$, for all $w\in W$ and $w_h \in W_h$, there holds
\[
 |u_h|_{S_p} + |z_h|_{S_a}  \lesssim (c_{a,\gamma}+c_{a,\gamma,\epsilon}) h^k
|u|_{H^{k+1}(\Omega)} + \epsilon(h) \|z_h\|.
\]
\end{lemma}
\begin{proof}
Let $\xi_h = \pi_V u - u_h$ and $\zeta_h = \pi_W z - z_h$. By the
definition \eqref{FEM} there holds
\[
|\xi_h|_{S_p}^2+|\zeta_h|_{S_a}^2   = s_p(\xi_h,\xi_h)+s_a(\zeta_h,\zeta_h) = a_h(\xi_h,\zeta_h) +
s_a(\zeta_h,\zeta_h) -  a_h(\xi_h,\zeta_h) + s_p(\xi_h,\xi_h).
\]
Using now the Galerkin orthogonality of $a_h(\cdot,\cdot)$,
\eqref{galortho}, we have
\[
|\xi_h|_{S_p}^2+|\zeta_h|_{S_a}^2  = a_h(\pi_V u - u,\zeta_h) +
s_a(\pi_W z,\zeta_h) -  a_h(\xi_h,\pi_W z - z) + s_p(\pi_V u-u, \xi_h).
\]
Observing that $z = \pi_W z = 0$ this reduces to
\[
|\xi_h|_{S_p}^2+|\zeta_h|_{S_a}^2  = a_h(\pi_V u - u,\zeta_h) + s_p(\pi_V u-u,\xi_h).
\]
We conclude by applying the continuity \eqref{cont1}
\[
|\xi_h|_{S_p}^2+|\zeta_h|_{S_a}^2 \leq \|u - \pi_V u\|_+
(c_a|\zeta_h|_{S_a} + \epsilon(h)\|z_h\|)+ |u-\pi_V u|_{S_p}|\xi_h|_{S_p} 
\]
followed by
\begin{align*}
|\xi_h|^2_{S_p} + |\zeta_h|^2_{S_a} &\leq (c_a^2+1)\|u
- \pi_V u\|^2_+ + \epsilon(h)^2 \|z_h\|^2 + |u-\pi_V u|^2_{S_p}\\
& \leq c_{a,\gamma}^2  (1 + c_a^2)  h^{2k}
|u|^2_{H^{k+1}(\Omega)}+ \epsilon(h)^2 \|z_h\|^2.
\end{align*}
The second result follows by adding and subtracting $\pi_V u$,
observing that here 
$\pi_W z = 0$, applying a triangle inequality and then \eqref{approx1}
on $|\pi_V u|_{S_p} = |\pi_V u - u|_{S_p}$. 
\end{proof}\\
We may now prove the main result which is optimal convergence in the
$L^2$ and the $H^1$ norms.
\begin{theorem}\label{main}
Assume that \eqref{forward} and \eqref{adjoint} are well-posed with
exact solutions $u$ and $z$ satisfying \eqref{smooth}. Assume that the forms of
   \eqref{FEM} and the operators $\pi_V$, $\pi_W$ 
    are such that \eqref{galortho}--\eqref{H2} are satisfied and that $h$ is so small that
\begin{equation}\label{eps_cond}
c_{a,\gamma,\Omega,} h\, \epsilon(h) \leq \frac1{6},
\end{equation}
where $c_{a,\gamma,\Omega}$ depends on the constants of the inequalities
 \eqref{smooth} and \eqref{approx1}--\eqref{H2}.
Then \eqref{FEM} admits a unique discrete solution $u_h,\, z_h$ that satisfies
\[
\|u - u_h\| + h \|\nabla (u - u_h)\| + \|z_h\| \leq C_{a,\Omega,\gamma} h^{k+1} |u|_{H^{k+1}(\Omega)}
\]
and in particular
\begin{equation}\label{rigor_apriori}
\|u - u_h\| + h \|\nabla (u - u_h)\| + \|z_h\| \leq C_{a,\Omega,\gamma}h^{2}\|f\|.
\end{equation}
\end{theorem}
\begin{proof}
Let $\varphi$ be the solution of \eqref{adjoint} with $g=u-u_h$ and
$\psi$ the solution of \eqref{forward} with $f = z_h$. By
\eqref{smooth} there holds
$$\|\varphi\|_{H^2(\Omega)} \leq c_{a,\Omega} \|u - u_h\|
\mbox{ and }
\|\psi\|_{H^2(\Omega)} \leq c_{a,\Omega} \|z_h\|.$$
By
definition of the primal and dual problems and by \eqref{consist1},
 \eqref{consist2}, \eqref{galortho}, \eqref{cont1} and \eqref{cont2} there holds, 
\begin{align*}
\|u - u_h\|^2 &+ \|z_h\|^2= (u - u_h,\mathcal{L}^* \varphi) +
(\mathcal{L} \psi,z_h) = a_h(u - u_h,\varphi)+a_h(\psi,z_h)\\
& =  a_h(u - u_h, \varphi- \pi_W \varphi) + s_{a}(z_h,\pi_W \varphi) \\
&\quad + a_h(\psi- \pi_V \psi,z_h) - s_{p}(u -u_h,\pi_V \psi)\\
&\leq (c_a  \|u - \pi_V u\|_{\mathcal{L}} + c_a |u_h - \pi_V u|_{S_p}+\epsilon(h) \|u - u_h\|) \|\varphi -
\pi_W \varphi\|_* \\ &
\quad+ (c_a|z_h|_{S_a} + \epsilon(h)\|z_h\|) \|\psi- \pi_V \psi\|_+ \\ & \quad +|z_h|_{S_a} |\pi_W \varphi|_{S_a}+|u_h-u|_{S_p} |\pi_V \psi|_{S_p} .
\end{align*}
First we observe that by \eqref{approx1}, \eqref{approx2} and \eqref{smooth}
\begin{align*}
\epsilon(h) \|u - u_h\| \|\varphi -
\pi_W \varphi\|_* & + \epsilon(h)\|z_h\|\|\psi- \pi_V \psi\|_+ \\ & \leq
c_{a,\gamma,\Omega} h (\epsilon(h) \|u - u_h\|^2 +  \epsilon(h)\|z_h\|^2).
\end{align*}
Then by Lemma \ref{stab_conv} and the upper bounds \eqref{H2} we have
\begin{align*}
 |z_h|_{S_a} |\pi_W \varphi|_{S_a}& +|u_h-u|_{S_p} |\pi_V \psi|_{S_p} \\ &
\lesssim
((c_{a,\gamma} + c_{a,\gamma,\epsilon}) h^k |u|_{H^{k+1}(\Omega)} +
\epsilon(h)\|z_h\|) c_{a,\gamma,\Omega} h (\|u-u_h\| + \|z_h\|).
\end{align*}
Using the two previous bounds and an arithmetic-geometric inequality we have
\begin{multline*}
(1 - 3 c_{a,\gamma,\Omega} h \epsilon(h))( \|u - u_h\|^2 +\|z_h\|^2)\\
\leq C_{a,\gamma} h^{k+1} |u|_{H^{k+1}(\Omega)}
(|\varphi|_{H^2(\Omega)}+|\psi|_{H^2(\Omega)}).
\end{multline*}
Using equation \eqref{smooth}, the result for the $L^2$-norm follows
provided $h$ satisfies \eqref{eps_cond}. The result for the
$H^1$-norm follows using a global inverse inequality on the discrete
error and then the $L^2$-norm error estimate.
\[
\|\nabla (u - u_h)\| \leq \|\nabla (u - \pi_V u)\|+ \|\nabla (\pi_V u
- u_h)\| \lesssim h^k |u|_{H^{k+1}(\Omega)} +  h^{-1} \|\pi_V u
- u_h\|. 
\]
The existence of a unique solution to \eqref{FEM} is a consequence of \eqref{rigor_apriori}. Well-posedness of \eqref{forward} means that
$f=0$ implies $u=0$, but then by \eqref{rigor_apriori} $u_h = z_h =
0$, which shows that the matrix is invertible.
\end{proof}\\
The optimal convergence of the stabilisation terms follows. 
\begin{corollary}
Under the assumptions of Lemma \ref{stab_conv} and Theorem \ref{main}
there holds
\[
|\pi_Vu - u_h|_{S_p}  + |\pi_W z - z_h|_{S_a}\lesssim c_{s,\epsilon} h^k
|u|_{H^{k+1}(\Omega)} + O(h^{k+1}).
\]
\end{corollary}
\begin{proof}
Immediate consequence of Lemma \ref{stab_conv} and Theorem \ref{main}.
\end{proof}
\begin{remark}
The need to control a low order contribution of the dual solution
$z_h$ above usually comes from oscillation of data. Either in the form of
stabilisation terms that do not account for oscillation within the
element or error in the numerical quadrature.
\end{remark}

In case a G\aa rdings inequality holds for
\eqref{FEM} and $s_a(\cdot,\cdot)\equiv s_p(\cdot,\cdot)$ the
$H^1$-error can be recovered without using inverse inequalities as stated below.
\begin{corollary}
Assume that for the bilinear form $a(\cdot,\cdot)$ there exists $\lambda \in \mathbb{R}$ such that
\[
\|\nabla v_h\|^2 - \lambda \|v_h\|^2 \lesssim a_h(v_h,v_h) +s_p(v_h,v_h) 
\]
and that $s_a(\cdot,\cdot)\equiv s_p(\cdot,\cdot)$.
Then
\[
\|\nabla(u - u_h)\| \lesssim h^k |u|_{H^{k+1}(\Omega)}.
\]
\end{corollary}
\begin{proof}
Similar to proof of Lemma \ref{stab_conv} and therefore only sketched.
Let $\xi_h:= \pi_V u -u_h$.
It follows by G\aa rdings inequality that 
\[
\|\nabla \xi_h\|^2 \lesssim a_h(\xi_h,\xi_h) + \lambda \|\xi_h\|^2 +s_p(\xi_h,\xi_h).
\]
Using Galerkin orthogonality we have
\[
a_h(\xi_h,\xi_h) = a_h(\pi_V u - u, \xi_h)+s_a(z_h,\xi_h)
\]
and the rest follows as in Lemma \ref{stab_conv} by \eqref{cont1},
\eqref{approx1} and using the known convergences of  Lemma
\ref{stab_conv} and Theorem \ref{main}.
\end{proof}

\section{Stabilisation methods}\label{sec:stabilisation}
To fix the ideas let $\mathcal{L}$ be a second order elliptic
operator on conservation form,
\begin{equation}\label{ellipt_oper}
\mathcal{L} u := -\mu \Delta u + \nabla \cdot(\beta   u) + c u.
\end{equation}
Here $\mu\in \mathbb{R}^+$, ${\bf{\beta}} \in [C^{2}(\Omega)]^2$ is a non-solenoidal velocity
vectorfield and $c \in C^{1}(\Omega)$. Formally, the corresponding
bilinear form writes
\begin{equation}\label{formal_weak}
a(u,v) := (\mu \nabla u,\nabla v) + (\nabla \cdot(\beta   u) + c
u,v),\quad u,v \in H^1(\Omega).
\end{equation}
\textcolor{black}{
The continuities \eqref{cont1} and \eqref{cont2} suggest the
following design criteria on
the stabilisation operators.
}
\begin{equation}\label{abstract_stab1}
\inf_{w_h \in V_h} \|h(\mathcal{L} v_h - w_h)\|^2_h +
\|h^{\frac12} \mu^{\frac12} \jump{\nabla v_h
  \cdot n_F}\|^2_{\mathcal{F}_{int}} \lesssim s_p(v_h,v_h),
\end{equation}
\begin{equation}\label{abstract_stab2}
\inf_{w_h \in
  V_h} \|h(\mathcal{L}^* x_h - w_h)\|_h^2 + \|h^{\frac12}
\mu^{\frac12} \jump{\nabla x_h
  \cdot n_F}\|^2_{\mathcal{F}_{int}} \lesssim s_a(x_h,x_h)
\end{equation}
at least up to a non-essential low order perturbation. \textcolor{black}{If we neglect
terms due to boundary conditions we may apply an integration by parts in the left hand side of \eqref{cont1},
leading to
\[
a_h(v - \pi_V v, x_h) = \left< u - \pi_V v, \jump{\mu \nabla
    x_h \cdot n_F}\right>_{\mathcal{F}_{int}} + (v - \pi_V v, \mathcal{L}^* x_h)_h.
\]
Using a suitable weighting in $h$ and applying the Cauchy-Schwarz inequality
justifies \eqref{abstract_stab2}. The function $w_h$ may be included
provided the interpolant $\pi_V$ has suitable orthogonality
properties. It can be noted that one may also construct the
interpolant with orthognality properties on the element faces, so that the influence of the
gradient jump term may be reduced, we will not pursue this possibility
herein.}
The choice $w_h = 0$ in the first term in the left hand side of
\eqref{abstract_stab1} results in a least squares term on the (homogeneous)
residual over the element. It follows that the stabilisation relies on
two mechanisms: $L^2$-control of the element residual and
$L^2$-control of the gradient jumps over element edges. If higher order
differential equations are considered, jumps of higher derivatives
must be added. The design
criterion \eqref{abstract_stab1}-\eqref{abstract_stab2} makes it straightforward to adapt the analysis below to a
range of stabilisation methods, such as Galerkin least squares,
orthogonal sub-scales, continuous interior penalty or
discontinuous Galerkin methods. For the discontinuous Galerkin method
the penalty must act on the jump of $u_h$ itself and on the
jump of the normal gradient. In all cases however the jumps of the
gradient must be penalised, or an equivalent stabilisation operator
introduced. It therefore seems natural to consider two
stabilisations in more detail, first the GLS-stabilisation
combined with gradient penalty and then a CIP-stabilisation purely based
on penalty on jumps of derivatives of the approximate solution. We introduce the stabilisation operators
\begin{equation}\label{stab_primal}
s_p(u_h,v_h) := s_{p,GLS}(u_h,v_h) + s_{cip}(u_h,v_h)
\end{equation}
and 
\begin{equation}\label{stab_adjoint}
s_a(z_h,w_h):= s_{a,GLS}(z_h,w_h) + s_{cip}(u_h,v_h)
\end{equation}
where 
\[
s_{p,GLS}(u_h,v_h) := (\gamma_{GLS} h^2 \mathcal{L} u_h, \mathcal{L} v_h)_h ,
\]
\[
s_{a,GLS}(z_h,w_h) := (\gamma_{GLS} h^2 \mathcal{L^*} z_h, \mathcal{L^*} w_h)_h 
\]
and
\begin{equation}\label{jumpstab1}
s_{cip}(u_h,v_h) := \sum_{F \in\mathcal{F}_{int}} \int_F (h_F \gamma_{1,F}
\jump{\nabla u_h} \cdot \jump{\nabla v_h} + h_F^3 \gamma_{2,F}\jump{\Delta u_h} \jump{\Delta v_h}) ~\mbox{d}x.
\end{equation}
Here 
$\jump{\nabla u_h}\vert_F$ and $\jump{\Delta u_h}\vert_F$ denotes the
jump of the gradient and the Laplacian respectively over the
face $F$. 
Note that for smooth $u$, $s_{cip}(u,v_h)=0$ and hence $s_p(u,v_h) = s_{p,GLS}(u,v_h) =
(f,\gamma_{GLS} h^2 L v_h)_h$ showing that $s_p(u,v_h)$ is known. The abstract analysis typically holds for the parameter choices
$\gamma_{GLS}>0,\,\gamma_{1,F}>0,\, \gamma_{2,F}=0$ or
$\gamma_{GLS}=0,\,\gamma_{1,F}>0,\, \gamma_{2,F}>0$.
Note that
the matrix stencil for finite element methods remains the same for both approaches, and therefore
the CIP-method seems more appealing in this
context. Eliminating the Galerkin least squares term also reduces the
computational effort since the same stabilisation is used for the
primal and adjoint solution. If on the other hand a $C^1$-continuous approximation space
is used, the jumps of the gradients may be omitted and the GLS
stabilisation might prove competitive, since integrations on the faces
may then be avoided. Below we will only consider the case where
$V_h=W_h := X_h^k$ or some subset thereof, which will then be defined
in each case.
\subsection{Galerkin-least-squares stabilisation}
The GLS-method is one of the most popular
stabilised methods. To fix the ideas we
will assume that the
problems \eqref{forward} and \eqref{adjoint} are subject to
homogeneous Dirichlet
conditions and well-posed, with $f\in L^2(\Omega)$. For the readers
convenience we detail the Lagrangian \eqref{Lagrange} in this particular case
\begin{multline}
\L(u_h,z_h):= \frac12 \|\tau^{\frac12} (\mathcal{L} u_h -
f)\|_h^2  + \frac12 s_{cip}(u_h,u_h) \\ - \frac12
\|\tau^{\frac12} \mathcal{L^*} z_h\|_h^2 - \frac12 s_{cip}(z_h,z_h)- a(u_h,z_h) + (f,z_h).
\end{multline}
We let $V_h=W_h:= X_h^k \cap H^1_0(\Omega)$
The optimality conditions then write, find $u_h,z_h \in V_h \times W_h$
\begin{align} \label{GLS}
a(u_h,w_h) + s_a(z_h,w_h)& =
(f,w_h) \\ \nonumber
a(v_h,z_h)- s_p(u_h,v_h)  &
=- s_p(u,v_h) =- (f,\tau \mathcal{L} v_h)_h
\end{align}
for all $v_h,w_h \in V_h \times W_h$. Here $\gamma_{GLS} h^2
:= \tau>0$ and $\gamma_{1,F} \sim \mu$, $\gamma_{2,F} = 0$. We
assume that the physical parameters are all order unity for simplicity.
Observe the nonstandard structure of the stabilisation terms and that
the formulation is consistent for $u$ the exact solution of
\eqref{forward} and $z=0$. We will now prove that the assumptions of
Proposition \ref{stab_conv} and Theorem \ref{main} are satisfied for the formulation
\eqref{GLS}.

We define the following semi-norms
\begin{equation}\label{dualnorm}
\|v\|_+ := \|v\|_* := \|\tau^{-\frac12} v\|+ \|\mu^{\frac{1}{2}}h^{-\frac12} v\|_{\mathcal{F}_{int}},
\end{equation}
and
\[
\|v\|_\mathcal{L} :=|x|_{S_p} := \|\tau^{\frac12} \mathcal{L} x\|_h + s_{cip}(x,x)^{\frac12}
\mbox{ and }
|x|_{S_a} := \|\tau^{\frac12} \mathcal{L}^* x\|_h + s_{cip}(x,x)^{\frac12},
\]
defined for $x \in H^2(\Omega) + V_h$.
Let $\pi_V$ and $\pi_W$ be defined by the Lagrange
interpolator $i_h$ (or any other $H^2$-stable interpolation operator
that satisfies boundary conditions), and note that by \eqref{approx} 
we readily deduce the following approximation results for smooth enough
functions $u$
\[
\|u - \pi_V u\|_+ + \|u - \pi_V u\|_{\mathcal{L}} + |u - \pi_V
u|_{S_p}+ \|u - \pi_W u\|_
* +|u - \pi_W u|_{S_a} \leq c_{a,\gamma} h^{k} |u|_{H^{k+1}(\Omega)}
\]
and, by $H^2$-stability of the interpolation operator,
\[
|\pi_V v|_{S_p}  \leq c_{\gamma,a} h \|v\|_{H^2(\Omega)},\quad  |\pi_W w|_{S_a} \leq c_{\gamma,a} h \|w\|_{H^2(\Omega)}, \quad\forall v,w \in H^2(\Omega).
\]
This shows that \eqref{approx1} and \eqref{approx2} holds. It then
only remains to show the continuities \eqref{cont1} and \eqref{cont2}.
First we show the inequality \eqref{cont1}
For the second order elliptic problem we note that after an
integration by parts and Cauchy-Schwarz inequality,
\begin{align*}
a_h(v - \pi_V v, x_h) &=  \sum_{F\in \mathcal{F}_{int}} \left<u - \pi_V
  u, \jump{\mu \nabla x_h \cdot
    n_F} \right>_F +  (u - \pi_V u, \mathcal{L}^* x_h)_h \\
&\leq
\|u - \pi_V u\|_+ |x_h|_{S_a}.
\end{align*}
Similarly, to prove \eqref{cont2} we integrate by parts in the opposite
direction in the second order operator and obtain
\begin{align*}
a_h(u - u_h,y - \pi_W y) &= \sum_{K \in \mathcal{T}_h} (\mathcal{L}
(u-u_h), y - \pi_W y)_K \\ & \qquad  +
\sum_{F\in \mathcal{F}_{int}} \left< \jump{\mu \nabla u_h \cdot n_F} ,
  y - \pi_W y \right>_F \\
& \|y - \pi_W y\|_* (\|u - \pi_V u\|_{\mathcal{L}} + |u_h - \pi_V u|_{S_p})
\end{align*}
\begin{remark}
Note that for the GLS-method $\epsilon(h)=0$ in \eqref{cont1} and \eqref{cont2}. This follows from the fact that the
whole residual is considered in the stabilisation term. This nice
feature however only holds under exact quadrature. When the integrals are
approximated, the quadrature error may give rise to
oscillation terms from data that introduces a non-zero
contribution to $\epsilon(h)$.
\end{remark}
\subsection{Continuous interior penalty}
Since in this case we must account for possible oscillation of the physical
coefficients we postpone the detailed analysis to the examples below
and here only discuss the general principle.
In this case we use $\gamma_{GLS}=0,\,\gamma_{1,F}>0,\,
\gamma_{2,F}>0$
in the general expressions for the stabilisation \eqref{stab_primal}
and \eqref{stab_adjoint}.
The parameters $\gamma_{i,F}$, $i=1,2$ are stabilisation coefficients, the
form of which will be problem specific and will be given for each
problem below. 
The key observation is that the following discrete approximation
result holds for suitably chosen $\gamma_{i,F}$ in $s_{cip}(\cdot,\cdot)$ (see \cite{Bu05, BFH06})
\begin{equation}\label{oswald_int}
\|h^{\frac12}(\beta_h \cdot \nabla u_h - I_{os} \beta_h \cdot \nabla
u_h)\|^2+ \sum_{K} \|h \mu (\Delta u_h - I_{os}
\Delta u_h)\|^2 \leq  s_{cip}(u_h,u_h).
\end{equation}
Here $\beta_h$ is some piecewise affine interpolant of the velocity
vector field $\beta$ and $I_{os}$ is the quasi-interpolation operator defined in each
node of the mesh as a straight average of the function values from
simplices sharing that node (see \cite{BFH06}). For example,
\[
(I_{os} \Delta u_h)(x_i) = N_i^{-1} \sum_{\{K : x_i \in K\}} \Delta u_h(x_i)\vert_K,
\]
with $N_i := \mbox{card} \{K : x_i \in K\}$. Using \eqref{oswald_int}
one may prove that
\begin{equation}\label{discrete_interp}
\inf_{v_h \in V_h} \|\mathcal{L} u_h - v_h\|_h \lesssim
s_{cip}(u_h,u_h)^{\frac12} + \epsilon(h) \|u_h\|.
\end{equation}
It immediately follows that \eqref{abstract_stab1} and
\eqref{abstract_stab2} are satisfied.
We will leave the discussion of \eqref{cont1} - \eqref{approx2} and
\eqref{discrete_interp} to the applications below, giving the explicit
form for $\epsilon(h)$ for each case.
Here we instead proceed with an abstract analysis, assuming that
all physical parameters are of order $O(1)$.
We choose 
$\pi_V$ and $\pi_W$ as the $L^2$-projection in order to exploit
orthogonality to ``filter'' the element residual.Let $\|\cdot\|_+$ and
$\|\cdot\|_*$ have the same definition as in the GLS case and define
\begin{equation}\label{op_norm}
\|u\|_{\mathcal{L}} := \|h \mathcal{L} u\|_h + \|h^{\frac12} \mu^{\frac12}
\jump{\nabla u \cdot n_F}\|_{\mathcal{F}_{int}}+  \epsilon(h) \|u\|.
\end{equation}
Then
we proceed similarly as for GLS, but we use the orthogonality of the
$L^2$-projection, ignoring here the contribution from boundary terms. It then follows using the
orthogonality of the projection that formally
\begin{align*}
a_h(v - \pi_V v, x_h) &=  \sum_{F\in \mathcal{F}_{int}} \left<u - \pi_V
  u, \jump{\mu\nabla x_h \cdot
    n_F} \right>_F +  (u - \pi_V u, \mathcal{L}^* x_h - w_h)_h \\
&\leq
\|u - \pi_V u\|_+ (|x_h|_{S_a} + \epsilon(h) \|x_h\|).
\end{align*}
Similarly, to prove \eqref{cont2} we integrate by parts in the opposite
direction in the second order operator and use the $L^2$-orthogonality
to obtain
\begin{align*}
a_h(u - u_h,y - \pi_W y) &= \sum_{K \in \mathcal{T}_h} (\mathcal{L}
(u-u_h), y - \pi_W y)_K \\ & \qquad  +
\sum_{F\in \mathcal{F}_{int}} \left< \jump{\mu \nabla u_h \cdot n_F} ,
  y - \pi_W y \right>_F \\ &
= \sum_{K \in \mathcal{T}_h} (\mathcal{L}
(u-\pi_V u) + \mathcal{L}(\pi_V u - u_h) - w_h, y - \pi_W y)_K \\ & \qquad  +
\sum_{F\in \mathcal{F}_{int}} \left< \jump{\mu \nabla u_h \cdot n_F} ,
  y - \pi_W y \right>_F \\
 \leq \|y - \pi_W y\|_* &(\| h \mathcal{L} (u - \pi_V u)\|_{h} +
 |\pi_V u - u_h
 |_{S_p}+\epsilon(h) \|\pi_V u - u_h\|) \\
 \leq \|y - \pi_W y\|_* &(\|u - \pi_V u\|_{\mathcal{L}} + |u_h
 - \pi_V u|_{S_p}+\epsilon(h) \|u - u_h\|)
\end{align*}
The last inequality follows by adding and subtracting $u$ in the last
norm in the right hand side to obtain $\epsilon(h) \|\pi_V u - u +
u - u_h\|$. This term is then split using a triangular inequality and
the approximation error integrated in the $\|\cdot\|_\mathcal{L}$ term. 
To use the $L^2$-projection in this fashion we must impose the
boundary conditions weakly so that the boundary degrees of freedom are
included in $V_h$. For the GLS-method one has the choice between weak
and strong imposition of boundary condition. In the next section we
will discuss how weakly imposed boundary conditions are
included in the formulation \eqref{FEM}.
\subsection{Imposition of boundary conditions}\label{subsec:bc}
To impose weak boundary conditions in this framework we propose a Nitsche type
method. However our formulation differs from the standard Nitsche boundary
conditions in several ways:
\begin{itemize}
\item[--] Both Dirichlet and Neumann conditions are imposed using
  penalty.
\item[--] There is no lower bound of the parameter for the imposition
  of Dirichlet type boundary conditions. This is related to the fact
  that the method never uses the coercivity of $a_h(\cdot,\cdot)$.
\item[--] Nitsche type boundary terms are added to $a_h(\cdot,\cdot)$ in
  order to ensure consistency and adjoint consistency, but the penalty
  is added to the operators $s_p(\cdot,\cdot)$ and $s_a(\cdot,\cdot)$,
  allowing for different boundary penalty for the primal and the
  adjoint. As we shall see below for some problems this is the only
  way to make the Nitsche formulation consistent.
\end{itemize}
If the primal and the dual problems have a  Dirchlet boundary condition
on $\Gamma_D$ this is imposed by 
\[
a_h(u_h,v_h) := a(u_h,v_h) -\left< \mu \nabla u_h \cdot n,v_h \right>_{\Gamma_D}-\left< \mu \nabla v_h \cdot n,u_h \right>_{\Gamma_D},
\]
where $a(\cdot,\cdot)$ is defined by equation \eqref{formal_weak} and by adding the boundary penalty term
\begin{equation}\label{Dir_penalty}
\int_{\Gamma_D} \gamma_D \mu h^{-1} u_h v_h ~\mbox{d}s
\end{equation}
to $s_p(\cdot,\cdot)$ and $s_a(\cdot,\cdot)$ with $\gamma_D>0$. In the
non-homogeneous case the suitable data is added to the right hand side
in the standard way.
For Neumann conditions on $\Gamma_N$ in both the primal and the
adjoint problems, these are introduced in the standard way in $a(u,v)$
with a suitable modification of the right hand side of \eqref{forward}.
No modification is introduced in $a_h(\cdot,\cdot)$ but the following
penalty is added to $s_p(\cdot,\cdot)$ and $s_a(\cdot,\cdot)$ with $\gamma_N>0$
\begin{equation}\label{Neum_penalty}
\int_{\Gamma_N}  \gamma_N h \nabla u_h \cdot n \nabla v_h \cdot n ~\mbox{d}s.
\end{equation}
If the boundary
conditions for $u$ are non-homogeneous the usual data contributions are introduced
in the right hand side $-s_p(u,w_h)$.

As mentioned in the introduction the semi-norm $|\cdot|_S$ can be
a norm in certain situations so that the partial
coercivity \eqref{partial_coerc} implies the well-posedness of
the linear system \eqref{FEM}. In the following Proposition we discuss
some basic sufficient conditions for the matrix to be invertible in
the case of piecewise affine approximation spaces. For
particular cases other geometric arguments may prove fruitful as we
shall see in the second example below.
\begin{proposition}\label{invert_mat}
The kernel of the
linear system defined by \eqref{FEM} with the stabilisation
\eqref{jumpstab1} has dimension at most $2(d+1)$ for $k=1$. 
The system \eqref{FEM} admits a unique solution if the boundary
conditions satisfy one of the following conditions:
\begin{enumerate}
\item Two non-parallel polygon sides subject to Dirichlet boundary
  conditions.
\item Two non-orthogonal polygon sides subject one to a Dirichlet
  boundary condition and the other a
  Neumann condition imposed using \eqref{Neum_penalty}.
\item $d$ non-parallel  polygon sides subject to
  Neumann conditions imposed using \eqref{Neum_penalty} and either $1
  \not \in V_h$ or there
  exists $v_h,w_h \in V_h$ such that $a_h(1,v_h) \ne 0$ and $a_h(w_h,1) \ne 0$.
\end{enumerate}
\end{proposition} 
\begin{proof}
It is immediate from \eqref{partial_coerc} that the kernel of the
system matrix of \eqref{FEM} can not be larger than the sum of the
dimensions of the
kernels of $s_{p}(\cdot,\cdot)$ and $s_a(\cdot,\cdot)$. For $s_{cip}(\cdot,\cdot)$ and $k=1$ the
kernel is identified as $[\mathbb{P}_1(\Omega)]^2$, with dimension $2(d+1)$.

To prove well-posedness of the linear system it is enough to prove
uniqueness, we assume that $f=0$ and prove that then $u_h \equiv z_h
\equiv 0$.

If Dirichlet boundary condition is imposed on a boundary then
the gradient must be zero in the tangential direction to this
boundary, since the tangents of two boundaries spans $\mathbb{R}^d$ we
conclude that $f=0$ in \eqref{FEM} implies $u_h = 0$ due to
\eqref{partial_coerc} and similarly $z_h = 0$ and the matrix is invertible.

In the second case, the function is zero on the Dirichlet
boundary and the gradient is zero in the
tangential directions of the Dirichlet boundary condition eliminating
$d$ elements in the kernel. The penalty on the Neumann boundary, being
non-orthogonal to the Dirichlet boundary, cancels the remaining
free gradient. The same argument leads to both $u_h=0$ and $z_h=0$.

For the third case we observe that the term \eqref{Neum_penalty}
acting on $d$ non-parallel polygon sides implies that $\nabla u_h = 0$
and well-posedness is then immediate by the remaining conditions.
\end{proof}
\begin{remark}
Observe that the Proposition \ref{invert_mat} holds for any bilinear
form $a(\cdot,\cdot)$, even strongly degenerate ones.
\end{remark}
\section{Applications}\label{sec:applications}
We will now give three examples of problems that enter the abstract
framework. 
The first two problems we introduce below have well-posed
primal and adjoint problems so that the above theory applies. For each
method we will propose a formulation and prove that the relations
\eqref{consist1}, \eqref{consist2} hold. We only consider the
CIP-method in the examples below, in the first example we detail the
dependence
of physical parameters in all norms and coefficients and chose
stabilisation parameters to allow for high P\'eclet number flows. 
Due to the use of the duality argument however the
present analysis is restricted  to the case of moderate
P\'eclet numbers. In the later
examples we assume that all physical parameters are unity and do not
track the dependence.
As suggested above we take $\pi_V \equiv
\pi_W \equiv \pi_L$. 
In each case we will detail the form of $\epsilon(h)$. In the last
case, the elliptic Cauchy problem,
the stability properties of the problem is strongly depending on the
geometry of the problem and the assumption of well-posedness does not
hold in general. We will nevertheless propose a method that satisfies
the assumptions of the general theory and then study its performance numerically.

\subsection{Nonsymmetric indefinite elliptic problems}
Our first examples consist of a convection-diffusion--reaction problem with
non-solenoidal velocity field as is the
case for reactive transport in compressible flow. We first consider
the case of homogeneous Dirichlet conditions where the analysis of
\cite{Schatz74} applies. Then we consider the
case where failure of the coercivity is due also to the boundary condition,
here we study a convection-diffusion equation
with homogeneous Neumann boundary conditions. We will detail only how
the analysis of this case differs from the Dirichlet case. For a
detailed analysis of the well-posedness of the continuous problems we
refer to \cite{Dro02, DV09} and for a finite element analysis in the case of
homogeneous Neumann conditions to \cite{CF88}. Recent work on
numerical methods for these problems have focused on finite volume
methods, see \cite{DGH03,CD11} or hybrid finite element/finite volume
methods \cite{KT12}.

\subsubsection{Reactive transport in compressible flow: Dirichlet conditions}\label{Dirichlet}
In combustion problems for example it is important to accurately
compute the transport of the reacting species in the compressible
flow. We suggest a scalar model problem of convection-diffusion type
with a linear reaction term $c u$, where the reaction can
have arbitrary sign.
\begin{align}\label{convdiff0}
\mathcal{L} u & = f \quad\mbox{  in }\Omega\\
u & = 0 \quad \mbox{ on } \partial \Omega.\nonumber
\end{align}
The dual adjoint takes the form
\begin{align}\label{convdiff0_dual}
\mathcal{L}^* z := -\mu \Delta z - \beta \cdot  \nabla  z + c z& = g \quad\mbox{  in }\Omega\\
z & = 0 \quad \mbox{ on } \partial \Omega.\nonumber
\end{align}
The variational formulation \eqref{forward} is obtained by taking $V =
 W:=
H^1_0(\Omega)$ and $a(\cdot,\cdot)$ defined by \eqref{formal_weak}.
We assume that $f, g \in L^2(\Omega)$, that both \eqref{convdiff0} and \eqref{convdiff0_dual} are
well posed in $H^1_0(\Omega)$ by the Fredholm alternative and that the smoothing property
\eqref{smooth} holds. See \cite{Dro02} for an analysis of existence
and uniqueness under weaker regularity assumptions on $\beta$ and $c$,
with $c \ge 0$. The below analysis can
also be carried out assuming less regularity, but the constraints on
the mesh-size for the error estimate to hold will be stronger. Recall
that the constants in the estimate
\eqref{smooth} also depends on the regularity of the coefficients. 
The discrete form of the bilinear form is given
by
\begin{equation}\label{bilin1}
a_h(u_h,v_h) := a(u_h,v_h) - \left< \nabla u_h \cdot
  n,v_h\right>_{\partial \Omega}- \left< \nabla v_h \cdot
  n,u_h\right>_{\partial \Omega} - \left<(\beta \cdot n)_- u_h,
  v_h\right>_{\partial \Omega}
\end{equation}
where $(\beta \cdot n)_{\pm} := \frac12 (\beta \cdot n \pm |\beta \cdot
n|)$. We define the approximation spaces $V_h = W_h := X_h^k$.
The stabilisation is chosen as
\begin{equation}\label{stab1p}
s_p(u_h,v_h):=s_{cip}(u_h,v_h) + s^-_{bc}(u_h,v_h)
\end{equation}
and
\begin{equation}\label{stab1a}
s_a(z_h,v_h):=s_{cip}(z_h,v_h) + s^+_{bc}(z_h,v_h)
\end{equation}
where $\gamma_{1,F} \sim (\mu+\|\beta_h\cdot n_F\|_{\infty,F} h_F)$
and $\gamma_{2,F} \sim \mu$ in
\eqref{jumpstab1} with $\beta_h$ the nodal interpolant of $\beta$
and
\begin{equation}\label{stab_bc}
s^\pm_{bc}(x_h,v_h) := \left<\mu h^{-1}\, x_h,v_h \right>_{\partial
  \Omega}+ \left<|(\beta \cdot n)_\pm| x_h,v_h \right>_{\partial \Omega}.
\end{equation}
\textcolor{black}{If only the low P\'eclet regime is considered the second term of
\eqref{stab_bc} is always dominated by the first and may therefore be omitted.}
\begin{proposition}(Existence of discrete solutions)
Let $k=1$. Then the formulation \eqref{FEM} with the bilinear form
\eqref{bilin1} and the stabilisation \eqref{stab1p}--\eqref{stab1a} admits a unique solution $u_h \in V_h$.
\end{proposition}
\begin{proof}
Immediate by Proposition \ref{invert_mat}.
\end{proof}\\
It is well known that the bilinear form \eqref{bilin1} satisfies the
consistency
relations \eqref{consist1} and \eqref{consist2} and that the
stabilisation \eqref{stab1p}-\eqref{stab1a} satisfies the upper bounds
\eqref{approx1}, \eqref{approx2} and \eqref{H2}. Now we define
the norms by
\[
\|v\|_+ := \|v\|_* := \|\mu^{\frac12} h^{-\frac12} v\|^2_{\mathcal{F}_{int}}  +
\|(\zeta_{Pe}+h^{-1}) v\| +\|h^{\frac12} \mu^{\frac12} \nabla v \|_{\partial \Omega}+
\|\beta_{\infty}^{\frac12} v\|_{\partial \Omega}.\]
Here $\zeta_{Pe}:= (\beta_{\infty}^{\frac12}
h^{-\frac12}+ \mu^{\frac12} h^{-1} + c_\infty^{\frac12})$ with
$\beta_{\infty}=\|\beta\|_{L^\infty(\Omega)}$ and $c_\infty := \|c\|_{L^{\infty}(\Omega)}$.
Also define
\begin{multline*}
\|v\|_{\mathcal{L}} := \|\mu^{\frac12} h \Delta v\|_h +
  \|\beta_{\infty}^{-\frac12} h^{\frac12} \beta \cdot \nabla v\| +
  \|c_{\infty}^{\frac12}  v\| + \|\mu^{\frac12} h^{\frac12}
  \jump{\nabla v \cdot n_F}\|_{\mathcal{F}_{int}} \\ + \|(\mu^{\frac12} h^{-\frac12} +
    \beta_{\infty}^{\frac12}) v\|_{\partial \Omega} +  \epsilon(h) \|v\|.
\end{multline*}
It is straightforward to show that
\[
\|u - \pi_V u\|_+ + \|u-\pi_Wu\|_* \lesssim (\zeta_{Pe} + h^{-1}) h^{k+1} |u|_{H^{k+1}(\Omega)}
\]
and  (for simplicity with $\epsilon(h)=0$)
\[
\|u - \pi_V u\|_{\mathcal{L}} \lesssim \zeta_{Pe} h^{k+1} |u|_{H^{k+1}(\Omega)}.
\]
It then only remains to prove the continuities \eqref{cont1} and
\eqref{cont2} to conclude that the Theorem \ref{main} holds.
\begin{proposition}
The bilinear form \eqref{bilin1} satisfies the continuities
\eqref{cont1} and \eqref{cont2} with $$\epsilon(h) 
\sim
h^{2} (|\beta|_{W^{2,\infty}}+ |c|_{W^{1,\infty}}).$$  
\end{proposition}
\begin{proof}
First we consider \eqref{cont1}.
After an integration by parts in $a(\cdot,\cdot)$ we have
\begin{multline*}
a_h(u - \pi_V u,x_h) = \sum_{F\in \mathcal{F}_{int}} \left<u - \pi_V
  u, \jump{\mu \nabla x_h \cdot
    n_F} \right>_F +  (u - \pi_V u, \mathcal{L}^* x_h)_h \\
- \left< u - \pi_V u, (\beta
  \cdot n)_+ x_h\right>_{\partial \Omega} = I+II+III.
\end{multline*}
Considering $I-III$ we find using Cauchy-Schwarz inequality
\[
I + III \leq \|u - \pi_V u\|_+ |x_h|_{S_a}.
\]
For $II$, using the discrete interpolation results \eqref{oswald_int},
the discrete commutator property, see \cite{Berto99}, and standard
approximation followed by an inverse inequality in the last term
\begin{align*}
II & =  (u - \pi_V u, -i_h\beta \cdot \nabla x_h + I_{os} (i_h\beta_h \cdot
\nabla x_h) - \mu \Delta x_h + I_{os} \mu \Delta x_h)_h \\ &+  (u - \pi_V
u, c x_h - i_h (c x_h) ) 
+ (u - \pi_V u, (\beta - i_h\beta) \cdot \nabla x_h) \\
& \leq \|u - \pi_V u\|_+(c_{os} |x_h|_{S_a} + c_i h^{2}
(|\beta|_{W^{2,\infty}} +|c|_{W^{1,\infty}} ) \|x_h\|).
\end{align*}
The second
continuity follows in a similar fashion, 
\begin{align*}
a_h(u - u_h,y - \pi_W y) &= (\mathcal{L}
(u-u_h), y - \pi_W y)_h +
\sum_{F\in \mathcal{F}_{int}} \left< \jump{\mu \nabla u_h \cdot n_F} ,
  y - \pi_W y \right>_F \\ & \quad + \left<  (\beta \cdot n)_- u_h,  y
  -\pi_W y\right>_{\partial \Omega}  + \left<u_h , \mu \nabla (y -
  \pi_W y) \cdot n \right>_{\partial \Omega}
\\ &
= I+II+III+IV.
\end{align*}
Considering first the term $I$ we get, with $\xi_h = \pi_V u - u_h$
\begin{align*}
I &= (\mathcal{L} (u-\pi_V u) + \mathcal{L} \xi_h, y - \pi_W
y) \\
&\lesssim \|y - \pi_W y\|_* (\|u - \pi_V u\|_{\mathcal{L}} + |\xi_h|_{S_p} + h^{2} (|
\beta|_{W^{2,\infty}(\Omega)}+|c|_{W^{1,\infty}(\Omega)}) \|u-u_h\|),
\end{align*}
where we used once again the inequalities
\[
(\mu \Delta \xi_h + \beta_ h \nabla \xi_h, y - \pi_W y) \leq c_{os} |\xi_h|_{S_p}
\|y - \pi_W y\|_*,
\]
\[
((\beta - i_h \beta) \cdot \nabla \xi_h, y - \pi_W y) \lesssim h^2
|\beta|_{W^{2,\infty}} (\|u - \pi_V u\| + \|u - u_h\|)
\|y - \pi_W y\|_*
\]
and, by the discrete commutator property,
\begin{align*}
((\nabla \cdot \beta +c) \xi_h&, y - \pi_W y)  = ((\nabla \cdot \beta
+ c) \xi_h -
i_h((\nabla \cdot \beta+c) \xi_h), y - \pi_W y) \\
& \lesssim h^2
(|\beta|_{W^{2,\infty}} + |c|_{W^{1,\infty}}) (\|u - \pi_V u\| + \|u - u_h\|)\|y - \pi_W y\|_*.
\end{align*}
For the second, third and fourth terms we have using the Cauchy-Schwarz
inequality, adding and subtracting $\pi_V u$ and recalling the form of the boundary penalty term,
\begin{multline*}
II + III+IV \lesssim (\|\mu^{\frac12} h^{\frac12} \jump{\nabla u_h
  \cdot n_F}\|_{\mathcal{F}_{int}} + \|(\mu^{\frac12} h^{-\frac12} +|(\beta \cdot n)_-|^{\frac12})
  u_h\|_{\partial \Omega}  )\|y - \pi_W y\|_*\\
\leq (\|u - \pi_V u\|_{\mathcal{L}} + |\pi_V u - u_h|_{S_p}) \|y - \pi_W y\|_*.
\end{multline*}
We conclude that
 the claim holds with $\epsilon(h) \sim  h^{2}
 (|\beta|_{W^{2,\infty}}+ |c|_{W^{1,\infty}})$.
\end{proof}
\begin{remark}
Note that if the physical parameters are constant, then the analysis
holds without restrictions on the mesh size in contrast to the
standard Galerkin analysis of \cite{Schatz74}. In this case, for $k=1$ the estimate takes
the simple form
\[
\|u - u_h\| \lesssim c_{a,\Omega} \zeta_{Pe}^2 h^4 |u|_{H^{2}(\Omega)}
\]
Assuming that $\beta_{\infty} \sim O(1),\, c_{\infty} = 0$ we get 
\[
\|u - u_h\| \lesssim c_{a,\Omega} \left(\frac{1}{h} + \frac{\mu}{h^2}\right) h^4 |u|_{H^{2}(\Omega)}
\]
Here the constant $c_{a,\Omega}$ typically is proportional to some
negative power of $\mu$, making the estimate valid only for moderate
P\'eclet numbers. If we assume that $c_{a,\Omega} = O(\mu^{-1})$ we
see that the quasi optimal convergence of order $h^{\frac32}$ is
obtained when $h^{\frac32} < \mu$. A more precise estimate for the
hyperbolic regime, showing that the estimate can not degenerate further
even for vanishing $\mu$ is the subject of the second part of this work \cite{part2}.
\end{remark}

\subsubsection{Transport in compressible flow: pure Neumann conditions}\label{Neumann}
We will now consider the convection--diffusion equation with
homogeneous Neumann conditions.
The main difficulty in this problem compared to the previous one is
that due to the homogeneous Neumann condition, the primal and dual
problems have different boundary conditions. The non-solenoidal
$\beta$ imposes special compatibility conditions on $g$ leading to
complications in the finite element analysis and additional
stability issues for the discrete solution.  For this example we will
assume that all physical parameters are order one. After having presented
the problem and the method we propose, we first show that the
the discrete problem is well-posed for all mesh-sizes when piecewise
affine approximation is used. Then
we prove that the assumptions of Lemma \ref{stab_conv} and Theorem
\ref{main} are satisfied. Optimal error
estimates for the problem similar to that above are obtained after accounting for some
minor modifications needed to accomodate the compatibility conditions
particular to this problem. The problem reads
\begin{align}\label{convdiff_forward}
-\Delta u + \nabla \cdot (\beta u) & = f \quad\mbox{  in }\Omega\\
-\nabla u \cdot n + \beta \cdot n\, u & = 0 \quad \mbox{ on } \partial \Omega.\nonumber
\end{align}
The dual adjoint problem is formally written 
\begin{align}\label{convdiff_adjoint}
-\Delta z - \beta \cdot \nabla  z& = g \quad\mbox{  in }\Omega\\
-\nabla z \cdot n & = 0 \quad \mbox{ on } \partial \Omega. \nonumber
\end{align}
We assume that the  following compatibility conditions hold
\begin{equation}\label{compatibility}
\int_\Omega f ~\mbox{d}x = 0,\quad \int_\Omega g m ~\mbox{d}x = 0
\end{equation}
where $m \in H^2(\Omega)$, $m>0$  is the unique solution to the homogeneous form of the primal
problem
\begin{align}\label{mproblem}
-\Delta m + \nabla \cdot(\beta   m) & = 0 \quad\mbox{  in }\Omega\\
-\nabla m \cdot n + \beta \cdot n \, m & = 0 \quad \mbox{ on } \partial \Omega.\nonumber
\end{align}
under the additional constraint
\[
|\Omega|^{-1} \int_\Omega m ~\mbox{d}x = 1.
\]
Then the problems \eqref{convdiff_forward} and
\eqref{convdiff_adjoint}  are both well-posed by the Fredholm alternative.
Since we assume that the regularity estimate \eqref{smooth} holds, $m
\in C^0(\bar \Omega)$ and $\sup_{x \in \Omega} m =: M \in
\mathbb{R}^+$ and since $m>0$ we may introduce  $m_{min}:=inf_{\Omega}~
m >0$ (see \cite{CF88}).

The problem is cast in the form \eqref{forward} by setting $V:=
H^1(\Omega) \cap L^2_0(\Omega)$, where $L^2_0(\Omega)$ denotes the set
of functions with global average zero, and by taking
\[
a(u,v) := (\nabla u,\nabla v) - (u, \beta \cdot \nabla v).
\]
The finite element method
\eqref{FEM} is obtained by setting $V_h=W_h:= X_h^k \cap L^2_0(\Omega)$
\begin{equation}\label{bilin2}
a_h(u_h,v_h):=a(u_h,v_h)
\end{equation}
and the stabilisation operators
\begin{equation}\label{stab2}
s_x(\cdot,\cdot):=s_{cip}(\cdot,\cdot) + s_{bc,x}(\cdot,\cdot), \mbox{ with } x=a,p.
\end{equation}
$s_{cip}(\cdot,\cdot)$ is given by \eqref{jumpstab1} with
$\gamma_{i,F}:=1$, $i=1,2$. 
The boundary operators finally  are defined by
\[
s_{bc,p}(u_h,v_h):= \int_{\Omega} h (\nabla u_h \cdot n - 
\beta \cdot n u_h) (\nabla v_h \cdot n - 
\beta \cdot n v_h) ~\mbox{d}s
\]
for $k\ge2$ and 
\[
s_{bc,p}(u_h,v_h):= \int_{\Omega} h (\nabla u_h \cdot n - 
(i_h\beta) \cdot n u_h) ( \nabla v_h \cdot n - 
(i_h \beta) \cdot n v_h) ~\mbox{d}s
\]
for $k=1$,
 $s_{bc,a}(\cdot,\cdot)$ finally is given by equation
 \eqref{Neum_penalty}, with $\gamma_N \sim 1$. The boundary stabilisation
 operator $s_{bc,p}(\cdot,\cdot)$ for $k=1$ is only weakly consistent. It is straightforward
 to show that the inconsistency introduced by replacing $\beta$ by
 $i_h \beta$ is compatible with \eqref{approx1}. We omit the details here,
 but similar arguments are used below to prove the continuity \eqref{cont2}. 

\begin{proposition}(Existence of discrete solution)
Assume $k=1$ in the definition of $V_h$. Then there exists a unique solution $u_h$ to
the discrete problem \eqref{FEM}.
\end{proposition}
\begin{proof}
As before we assume $f=0$ and observe that 
\[
s_p(u_h,u_h) + s_a(z_h,z_h) = 0.
\]
This implies $z_h,u_h \in \mathbb{P}_1(\Omega)$. Since $\|\nabla z_h
\cdot n\|_{\partial \Omega} = 0$ and $z_h$ has zero average, we conclude that $z_h=0$. For $u_h$
there holds
\[
\|\nabla u_h \cdot n + (i_h \beta \cdot n) u_h\|_{\partial \Omega} = 0.
\]
Since  $\nabla u_h \cdot n$ is constant on every polyhedral side
$\Gamma$ of $\Omega$ so is $(i_h \beta \cdot n) u_h$. But since $(i_h
\beta \cdot n) u_h\vert_{\Gamma} \in \mathbb{P}_2(\Gamma)$ we conclude
that both $i_h \beta$ and $u_h$ must be constant. Since this is true
for all faces $\Gamma$ of $\Omega$, $u_h$ is a constant
globally. We conclude by recalling that zero average was imposed on
the approximation space.
\end{proof}\\
In case $k \ge 2$, we let the norms $\|\cdot\|_+,\, \|\cdot\|_*$ be defined by
\eqref{dualnorm} and $\|\cdot\|_\mathcal{L}$ by
\[
\|v\|_{\mathcal{L}} := \|\mathcal{L} v\|_h +\|h^{\frac12} \jump{\nabla v
\cdot n_F}\|_{\mathcal{F}_{int}} + \|h^{\frac12} \nabla v \cdot n \|_{\partial \Omega} +
\beta_{\infty} \|h^{\frac12} v \|_{\partial \Omega}.
\]
When
$k=1$ we let the norm $\|\cdot\|_+$ be defined by \eqref{dualnorm},
but define
\[
\|v\|_* := \|h^{-1} v\|+\|h^{-\frac12} v\|_{\mathcal{F}}+ \epsilon(h)
\|v\|_{\partial \Omega}
\]
and, assuming $h<1$,
\[
\|v\|_{\mathcal{L}} := \|\mathcal{L} v\|_h +\|h^{\frac12} \jump{\nabla v
\cdot n_F}\|_{\mathcal{F}_{int}} + \|h^{\frac12} \nabla v \cdot n \|_{\partial \Omega} +
(1 + \beta_{\infty}) \|v \|_{\partial \Omega}+\epsilon(h)\|v\|.
\]
For the projection
operators
$\pi_V$ and $\pi_W$ we once again choose the $L^2$-projection.
\begin{proposition}
The bilinear form \eqref{bilin2} satisfies the continuities
\eqref{cont1} and \eqref{cont2} with $$\epsilon(h) 
\sim
h^{2} |\beta|_{W^{2,\infty}(\Omega)}.$$
\end{proposition}
\begin{proof}
As before we integrate by parts in $a_h(\cdot,\cdot)$ to obtain
\begin{multline*}
a_h(u - \pi_V u,x_h) = \sum_{F\in \mathcal{F}_{int}} \left<u - \pi_V
  u, \jump{\nabla x_h \cdot
    n_F} \right>_F +  (u - \pi_V u, \mathcal{L}^* x_h)_h \\[3mm]
+ \left<u - \pi_V u, \nabla x_h \cdot n _{\partial \Omega} \right>_{\partial \Omega}   = I+II+III.
\end{multline*}
The treatment of terms $I$ and $II$ are identical to the Dirichlet case.
Term $III$ is bounded using Cauchy-Schwarz inequality, recalling that
the Neumann condition is penalised in $s_a(\cdot,\cdot)$
\[
III \leq \|u - \pi_V u\|_+ |x_h|_{S_a}.
\]
The second
continuity follows in a similar fashion. We write
\begin{align*}
a_h(u - u_h,y - \pi_W y) &= \sum_{K \in \mathcal{T}_h} (\mathcal{L}
(u-u_h), y - \pi_W y)_K \\ & \quad  +
\sum_{F\in \mathcal{F}_{int}} \left< \jump{\nabla u_h \cdot
    n_F} , y - \pi_W y \right>_F \\ & \quad - \left<  \nabla u_h \cdot
  n _{\partial \Omega} - (\beta \cdot n _{\partial \Omega}) u_h, y -
  \pi_W y\right>_{\partial \Omega} \\ &= I+II+III
\end{align*}
and observe that the treatment of terms $I$ and $II$ is analogous with
the Dirichlet case. For term $III$, when $k\ge 2$ we recall that $\nabla u_h \cdot n  - (\beta
\cdot n) u_h$ is penalised in $s_p(\cdot,\cdot)$ and we may conclude as
before using a Cauchy-Schwarz inequality
\[
III \lesssim \|u - \pi_V u\|_* ( |u_h -
\pi_V u|_{S_p} + \|u - \pi_V u\|_{\mathcal{L}}) .
\]
For the case
$k=1$ we must take care to handle the lack of consistency. Therefore we add and
subtract $i_h \beta$ and use the boundary condition on $u$ to get
\begin{align*}
III =& \left<  \nabla (u_h-u) \cdot n _{\partial \Omega}  - (i_h \beta
  \cdot n _{\partial \Omega}) (u_h-u), y -
  \pi_W y\right>_{\partial \Omega}
\\&+\left<  (i_h \beta - \beta)\cdot n _{\partial \Omega} (u_h-u), y -
  \pi_W y\right>_{\partial \Omega}.
\end{align*}
First we add and subtract $\pi_V u$, so that $u-u_h = u-\pi_V u + \xi_h$,
$\xi_h :=  \pi_V u  - u_h$ and split the scalar products with
Cauchy-Schwarz inequality. For the first term we immediately have
\begin{multline*}
\left<  \nabla (u_h-u) \cdot n _{\partial \Omega}  - (i_h \beta
  \cdot n _{\partial \Omega}) (u_h-u), y -
  \pi_W y\right>_{\partial \Omega} \\
\leq \|y - \pi_W y\|_* (\|u - \pi_V u\|_{\mathcal{L}} + |\xi_h|_{S_p}).
\end{multline*}
For the $\pi_V u - u$ part of the second term we observe that
\[
\left<  (i_h \beta - \beta)\cdot n _{\partial \Omega} (\pi_V
  u-u), y -
  \pi_W y\right>_{\partial \Omega} \lesssim \|y -
  \pi_W y\|_{\partial \Omega} h^2 |\beta|_{W^{2,\infty}(\Omega)}
  \|\pi_V u - u\|_{\partial \Omega}.
\]
Applying an element wise trace inequality in the $\xi_h$ part of the second term, we have
\[
\left<  (i_h \beta - \beta)\cdot n \xi_h, y -
  \pi_W y \right>_{\partial
  \Omega} \lesssim \|y -
  \pi_W y\|_{\partial
  \Omega} \|i_h \beta - \beta\|_{L^{\infty}(\partial \Omega)}
h^{-\frac12} \|\xi_h\|.
\]
Then we use the definition of the norms, in
particular that $ \|h^{-\frac12}(y - \pi_W y) \|_{\partial \Omega} \leq
\|y - \pi_W y\|_*$  to obtain
\[
III \lesssim \|y - \pi_W y\|_* (\|u - \pi_V u\|_{\mathcal{L}} +
|\xi_h|_{S_p} + h^{2} |\beta|_{W^{2,\infty}}\|\xi_h\|).
\]
Using the triangular inequality
$\|\xi_h + u - u\| \leq \|\pi_V u - u\| + \|u - u_h\|$, we obtain
\begin{align*}
III \lesssim   \|y - \pi_W y\|_* (\|u - \pi_V u\|_{\mathcal{L}}+
|\xi_h|_{S_p} + \epsilon(h)\|u-u_h\|)
\end{align*}
with $\epsilon \sim h^2 |\beta|_{W^{2,\infty}(\Omega)}$. The proof is complete.
\end{proof}
\begin{remark}
It is straightforward to verify that Lemma \ref{stab_conv} holds.
The assumptions of Theorem \ref{main}, however still are not satisfied since
we want to use the solutions of the problems $\mathcal{L}^* \varphi =
u - u_h$ and $\mathcal{L} \psi = z_h$, but the solution $\varphi$ 
will in general not exist since $u-u_h$ does not
satisfy the second compatibility condition of equation \eqref{compatibility}. Instead we will use $m$, the
solution of \eqref{mproblem} as weight, as suggested in \cite{CF88}, and solve the well-posed problem
\[
\mathcal{L}^* \varphi =
(u - u_h)/m.
\]
We may then write
\[
\|(u - u_h) m^{- \frac12}\|^2 + \| z_h \|^2 = (u - u_h,
\mathcal{L^*} \varphi) + (\mathcal{L} \psi, z_h)
\]
and proceed as in Theorem \ref{main}, now using the stability estimate
\[
|\varphi|_{H^2(\Omega)} \leq c_{a,\Omega} \|(u - u_h) m^{-1}\| \leq c_{a,\Omega}/m_{min}^{\frac12} \|(u - u_h) m^{-1/2}\|
\]
 to obtain
\[
\|(u - u_h) m^{- \frac12}\| + \| z_h \| \lesssim h^{k+1} |u|_{H^{k+1}(\Omega)}.
\]
For an estimate in the unweighted $L^2$-norm we observe that
\[
M^{-\frac12} \|u - u_h\| \leq \|(u - u_h) m^{-
  \frac12}\|.
\]
Convergence follows by Lemma \ref{stab_conv} and the modified
Theorem \ref{main}. Observe that the constants in $\epsilon(h)$ now
depends on the (unknown) minimum value of $m$.
\end{remark}
\begin{remark}
In practice the zero average condition can be imposed using Lagrange
multipliers. The above analysis holds for that case after minor modifications.
\end{remark}
\subsection{The Cauchy problem}
We consider the case of a Helmholtz-type problem where both the solution
itself and its normal gradient is specified on one portion of the
domain and the other portion is free. We let $\Gamma_V$ and
$\Gamma_W$ be connected subsets of $\partial \Omega$ such that $\partial \Omega :=
\bar \Gamma_V \cup \bar \Gamma_W$ and $\Gamma_V \cap \Gamma_W = \emptyset$. We will
consider the problem, $\kappa \in \mathbb{R}$, 
\begin{align}\label{Cauchy_f}
-\Delta u + \kappa u& = f \mbox{ in } \Omega\\
\nabla u \cdot n = u &= 0 \mbox{ on } \Gamma_V,\nonumber
\end{align}
with dual problem
\begin{align}\label{Cauchy_d}
-\Delta z + \kappa z & = g \mbox{ in } \Omega\\
\nabla z \cdot n = z &= 0 \mbox{ on } \Gamma_W. \nonumber
\end{align}
The weak formulations \eqref{forward} and \eqref{adjoint}  are
obtained by setting 
$$V:= \{v \in H^1(\Omega): v\vert_{\Gamma_V} = 0\}$$ 
and 
$$W:= \{v \in H^1(\Omega): v\vert_{\Gamma_W} = 0\}$$ 
and defining 
$$
a(u,v) := (\nabla u,\nabla v)+ \kappa (u,v), \quad \forall u \in V,\, v \in W.
$$
Note that both symmetry and G\aa rdings inequality fail in this case because the
functions in the bilinear form have to be taken in different spaces
and hence the choice $v=u$ is prohibited.

To design a suitable discrete formulation \eqref{FEM} for this problem we generalise the ideas of the Nitsche type weak imposition of
boundary condition. Observe that in
this case boundary conditions imposed using penalty in the standard
fashion can not be consistent for both the primal and the adjoint problems,
since the primal and dual solution are zero on different parts of the
boundary. It is therefore important in this case that two stabilisation
operators are used, one for the primal and one for the adjoint.  We define the approximation spaces $V_h = W_h := X_h^k$. We propose the bilinear form 
\begin{equation}\label{bilin_Cauchy}
a_h(u_h,v_h):=  (\nabla u_h,\nabla v_h)+ \kappa (u_h,v_h)- \left< \nabla v_h \cdot n, u_h\right>_{\Gamma_V}- \left< \nabla u_h \cdot n, v_h\right>_{\Gamma_W}
\end{equation}
and for the stabilisation we use
\begin{equation}\label{stab_Cauchy}
s_x(u_h,v_h) := s_{cip}(u_h,v_h) + s_{bc,x}(u_h,v_h),\quad
x= a,p
\end{equation}
where $s_{cip}(\cdot,\cdot)$ is given by \eqref{jumpstab1}, with
$\gamma_{F,i} \sim 1$, $i=1,2$, and 
\[
s_{bc,x}(u_h,v_h) := \int_X (h^{-1} u_h\, v_h + h \nabla u_h \cdot
n \nabla v_h \cdot n) ~\mbox{d}s, 
\]
where $X =  \Gamma_V$ 
  for $x=p$  and $X =  \Gamma_W$ 
  for $x=a$.
If some part of the boundary is equipped with Dirichlet or Neumann
boundary conditions this is imposed as described in Section
\ref{subsec:bc}.
\begin{proposition}(Existence of
  discrete solution for $k=1$)
Define \eqref{FEM} by the bilinear forms \eqref{bilin_Cauchy} and \eqref{stab_Cauchy}. Let
$k=1$ in $V_h$. Then there exists a unique solution $(u_h,z_h) \in [V_h]^2$ to \eqref{FEM}.
\end{proposition}
\begin{proof}
Let $f=0$, by \eqref{partial_coerc} there holds, $u_h, z_h \in
\mathbb{P}_1(\Omega)$ and $u_h\vert_{\Gamma_V} = \nabla u_h \cdot
n\vert_{\Gamma_V} = 0$ as well as $z_h\vert_{\Gamma_W} = \nabla z_h \cdot
n\vert_{\Gamma_W} = 0$, by which we conclude that the matrix is
invertible using case $(2)$ of Proposition \ref{invert_mat}.
\end{proof}\\
For the error analysis we once again choose the interpolants $\pi_V$
and $\pi_W$ to be the standard $L^2$-projection $\pi_L$.
We will now prove that the assumptions
\eqref{consist1}-\eqref{consist2} and \eqref{cont1}-\eqref{cont2}
are satisfied.
\begin{lemma}(Consistency of bilinear form) \label{lem_consist_Cauchy}
The bilinear form \eqref{bilin_Cauchy}  satisfies \eqref{consist1} and \eqref{consist2}.
\end{lemma}
\begin{proof}
By an integration by parts we see that for $u$ solution of \eqref{Cauchy_f}
\begin{align*}
(-\Delta u + \kappa u, v+v_h) & =  (\nabla u, \nabla  (v + v_h)) +
(\kappa u, v + v_h)  \\
- & \underbrace{\left<\nabla u
  \cdot n, v + v_h \right>_{\Gamma_W}}_{\mbox{ since } \nabla u\cdot n = 0 \mbox{ on }
\Gamma_V} - \underbrace{\left<\nabla (v + v_h)
  \cdot n, u\right>_{\Gamma_V}}_{\mbox{ since } u = 0 \mbox{ on }
\Gamma_V} = a_h(u, v+v_h).
\end{align*}
Similarly for $z$ solution of \eqref{Cauchy_d} consistency follows by
observing that
\[
(v+v_h, -\Delta z) = (\nabla  (v + v_h), \nabla z ) -\underbrace{ \left<\nabla z
  \cdot n, v + v_h \right>_{\Gamma_V}}_{\mbox{ since } \nabla z \cdot n = 0 \mbox{ on } \Gamma_W} - \underbrace{\left<\nabla (v + v_h)
  \cdot n, z\right>_{\Gamma_W}}_{\mbox{ since } z = 0 \mbox{ on } \Gamma_W}.
\]
\end{proof}\\
We define the norms $\|\cdot\|_+$ and $\|\cdot\|_*$ by
\[
\| v \|_+:=\|h^{-\frac12} v\|_{\mathcal{F}_{int}}+ \|h^{-1} v\| + \|h^{-\frac12}
v\|_{\Gamma_W} + \|h^{\frac12}
\nabla v \cdot n\|_{\Gamma_W},
\]
\[
\| v \|_*:= \|h^{-\frac12} v\|_{\mathcal{F}_{int}} + \|h^{-1} v\| +
\|h^{-\frac12}
v\|_{\Gamma_V} + \|h^{\frac12}
\nabla v \cdot n\|_{\Gamma_V}.
\]
and
$$
\|v\|_{\mathcal{L}} := \|h \mathcal{L} v\|_h + \|h^{\frac12} \jump{\nabla v \cdot n_F}\|_{\mathcal{F}_{int}}+\|h^{-\frac12}
v\|_{\Gamma_V} + \|h^{\frac12}
\nabla v \cdot n\|_{\Gamma_V}.
$$
It is straightforward to show \eqref{approx1} and \eqref{approx2}.
\begin{proposition}\label{prop_cont_Cauchy}
For $a_h(\cdot,\cdot)$ defined by \eqref{bilin_Cauchy}, the continuities \eqref{cont1} and \eqref{cont2} hold with $\epsilon(h) = 0$.
\end{proposition}
\begin{proof}
We proceed as before using an integration by parts in
\eqref{bilin_Cauchy} to obtain
\begin{align*}
a_h(v &- \pi_V v, x_h)  = \sum_{F \in \mathcal{F}_{int}} \left<v- \pi_V
v, \jump{\nabla x_h\cdot n_F}\right>_{F} +  (v - \pi_V v,
-\Delta x_h+ \kappa x_h)_h\\
& + \left< v - \pi_V v, \nabla x_h \cdot
n\right>_{\Gamma_W} - \left< \nabla(v - \pi_V v)\cdot n, x_h
\right>_{\Gamma_W} = I + II + III + IV.
\end{align*}
The first sum $I$ is upper bounded as before using the Cauchy-Schwarz
inequality and for the second sum, we use the orthogonality of the
$L^2$-projection, $(v - \pi_V v,\kappa x_h) = 0$ and the discrete
interpolation inequality \eqref{oswald_int} leading to
\[
I + II \lesssim \|u - \pi_V u\|_+ |x_h|_{S_a}.
\]
For the terms $III$ and $IV$ we note that by the definition of
$\|\cdot\|_+$ and $|\cdot|_{S_a}$ there also holds
\[
III+IV \leq \|u - \pi_V u\|_+ |x_h|_{S_a}.
\]
This ends the proof of \eqref{cont1}. The proof of \eqref{cont2} is
similar. Using integration by parts in the other direction we have
\begin{align*}
a_h(u - u_h , y - \pi_W y) &=  (-\Delta(u-u_h) + \kappa (u -
u_h),  y - \pi_W y)_h \\
& + \sum_{F \in \mathcal{F}_{int}} \left<\jump{\nabla u_h\cdot n_F},   y - \pi_W y\right>_F
+ \left< u - u_h, \nabla (y - \pi_W y) \cdot
n\right>_{\Gamma_V} \\
&+ \left< \nabla(u - u_h)\cdot n, y - \pi_W y
\right>_{\Gamma_V}  = I + II + III+IV.
\end{align*}
Using the same arguments as before, adding and subtracting $\pi_V u$
in all the terms in the left slot we have for the term $I$, using
$\xi_h = \pi_V u - u_h$,
\begin{align*}
I & = (-\Delta(u-\pi_V u)+\kappa (u -
\pi_V u),  y - \pi_W y)_h- 
(\Delta\xi_h - I_{os} \Delta\xi_h ,  y - \pi_W y)_h \\
& \qquad  \lesssim( \|u - \pi_V u\|_{\mathcal{L}}+ |\xi_h|_{S_p})\| y - \pi_W y\|_*.
\end{align*}
By the definition of the stabilisation operator and the fact that $u =
\nabla u \cdot n = 0$ on $\Gamma_V$ we may once again add and subtract
$\pi_V u$ in the terms $II$ and $III$ to obtain
\begin{align*}
II+III+IV &= \left<\jump{\nabla
    u_h\cdot n_F},   y - \pi_W y\right>_{\mathcal{F}_{int}}\\
& \qquad
- \left< u_h, \nabla ( y - \pi_W y) \cdot
n\right>_{\Gamma_V} - \left< \nabla  u_h\cdot n, y - \pi_W y
\right>_{\Gamma_V} \\
& \leq \| y - \pi_W y\|_* (\|u - \pi_V u\|_{\mathcal{L}} + |\xi_h|_{S_p}).
\end{align*}
By which we conclude.
\end{proof}
\begin{corollary}\label{cor:useless}
Assume that the problems \eqref{Cauchy_f} and \eqref{Cauchy_d} admit
unique solutions for which \eqref{smooth} holds. Then the conclusions
of Theorem \ref{main} holds for \eqref{FEM} defined by $V_h$,
\eqref{bilin_Cauchy} and \eqref{stab_Cauchy}.
\end{corollary}
\begin{proof}
In Lemma \ref{lem_consist_Cauchy} we verified the consistencies \eqref{consist1} and
\eqref{consist2}.
In Proposition \ref{prop_cont_Cauchy} we verified the continuities \eqref{cont1} and
\eqref{cont2}.
It is straightforward to verify that \eqref{approx1}-\eqref{approx2}
hold for $\pi_V=\pi_W=\pi_L$ and $s_p(\cdot,\cdot)$,
$s_a(\cdot,\cdot)$ defined by \eqref{stab_Cauchy} under the
assumptions on the mesh and the regularity assumptions on the solution.
\end{proof}
\begin{remark}
Admittedly Corollary \ref{cor:useless} is of purely academic interest
since the Cauchy problem under consideration in general is ill-posed,
with very weak stability properties. As we shall see in the numerical
section the method nevertheless returns useful approximations. 
An example of a sufficient condition for Theorem \ref{main} to result in a convergence
estimate, if $u$ is smooth, is that there exists $M \in \mathbb{R}^+$ and $s \in
\mathbb{R}$, with $s>-k$ such that $\|\varphi - \pi_V \varphi\|_*  \leq M h^s$, for
all $u - u_h$, with $\varphi$ the solution of \eqref{adjoint}, with $g = u-u_h$. 
The expected convergence rate in that case would be
\begin{equation}\label{cauchy_conjecture}
\|u-u_h\| \lesssim M^{\frac12} h^{(k+s)/2} |u|^{\frac12}_{H^{k+1}(\Omega)}.
\end{equation}
Unfortunately, no such stability estimates are known for the Cauchy
problem and a regularised adjoint would have to be considered. We refer
to \cite{BD10} for conditional stability estimates for the
problem \eqref{Cauchy_f} in general Lipschitz domains, leading to
logarithmic estimates and to
\cite{FM86, RHD99} and \cite{Lu95}
for other work on finite element methods on the Cauchy problem and some stability results under special geometrical assumptions.
\end{remark}
\section{Numerical investigations}\label{sec:numerical}
We will present numerical examples of convergence for a smooth exact
solution of the applications given above. For the computations we have
used FreeFEM++, \cite{freefem}. All problems will be set in
$\Omega:=(0,1)\times (0,1)$.  We use unstructured meshes with $2^N$
elements on each side, $N=3,\hdots,8$ and, drawing on our previous
experience of the CIP-method we fix the stabilisation
parameters to be $\gamma_{1,F}=0.01$ for piecewise affine approximation
and $\gamma_{i,F}=0.001$, $i=1,2$ for piecewise quadratic approximation. The
boundary penalty parameter is chosen to be $\gamma_{bc}=10$ for both
cases and both for Dirichlet and Neumann penalty terms. Let us remark that in particular for the ill-posed Cauchy
problem, an optimal choice of the stabilisation parameter can make a
big impact on the error on a fixed mesh, but does not appear to
influence the convergence behavior. For each example we plot
the error quantities estimated in Lemma \ref{stab_conv} and Theorem
\ref{main}. When appropriate we indicate the experimental convergence
order in parenthesis. We report the computational mesh for $N=5$ in
the left plot of figure \ref{mesh_velocities}.

\subsection{Convection--diffusion problems}
We consider an example given in \cite{CD11}
where, in \eqref{convdiff0}, the physical parameters are chosen as $\mu=1$, $c=0$,
\[
\beta:= -100 \left(\begin{array}{c} x+y \\ y-x \end{array} \right)
\]
(see the right plot of figure \ref{mesh_velocities}) and the exact solution is given by
\begin{equation}\label{exact_sol}
u(x,y) = 30x(1-x)y(1-y).
\end{equation}
This function satisfies homogeneous Dirichlet boundary conditions and
has $\|u\| = 1$. Note that $\|\beta\|_{L^\infty} = 200$ and $\nabla
\cdot \beta = -200$, making the problem strongly noncoercive with a
medium high P\'eclet number. The right hand side is then chosen as $\mathcal{L} u$
and in the case of (non-homogeneous) Neumann conditions, a suitable
right hand side is introduced to make
the boundary penalty term consistent. The
optimal convergence rate for the stabilising terms given in Lemma
\ref{stab_conv} is verified in all the numerical examples. 
\begin{figure}
{\centering
\includegraphics[width=6.5cm]{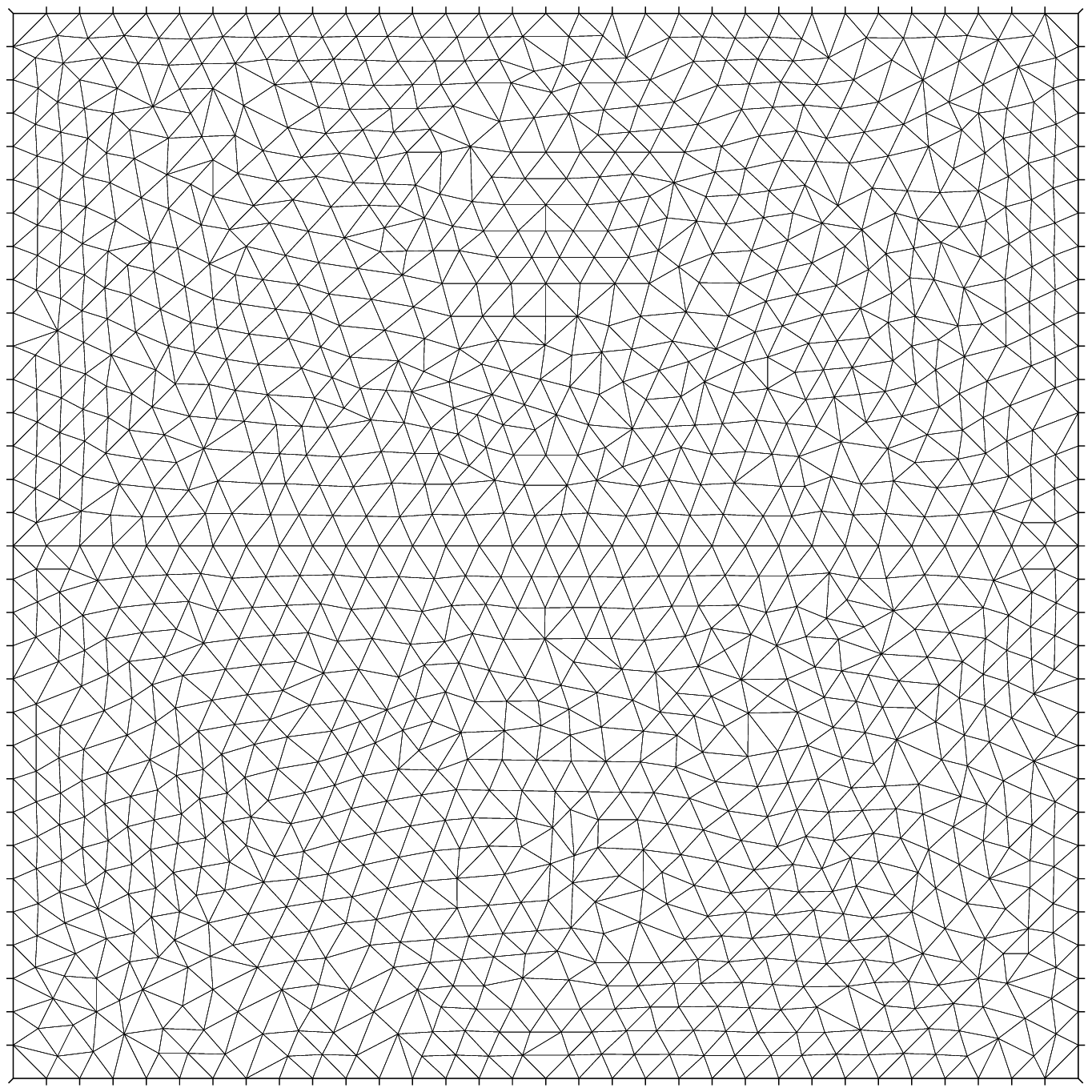}\hspace{-1cm}
\includegraphics[width=6.5cm]{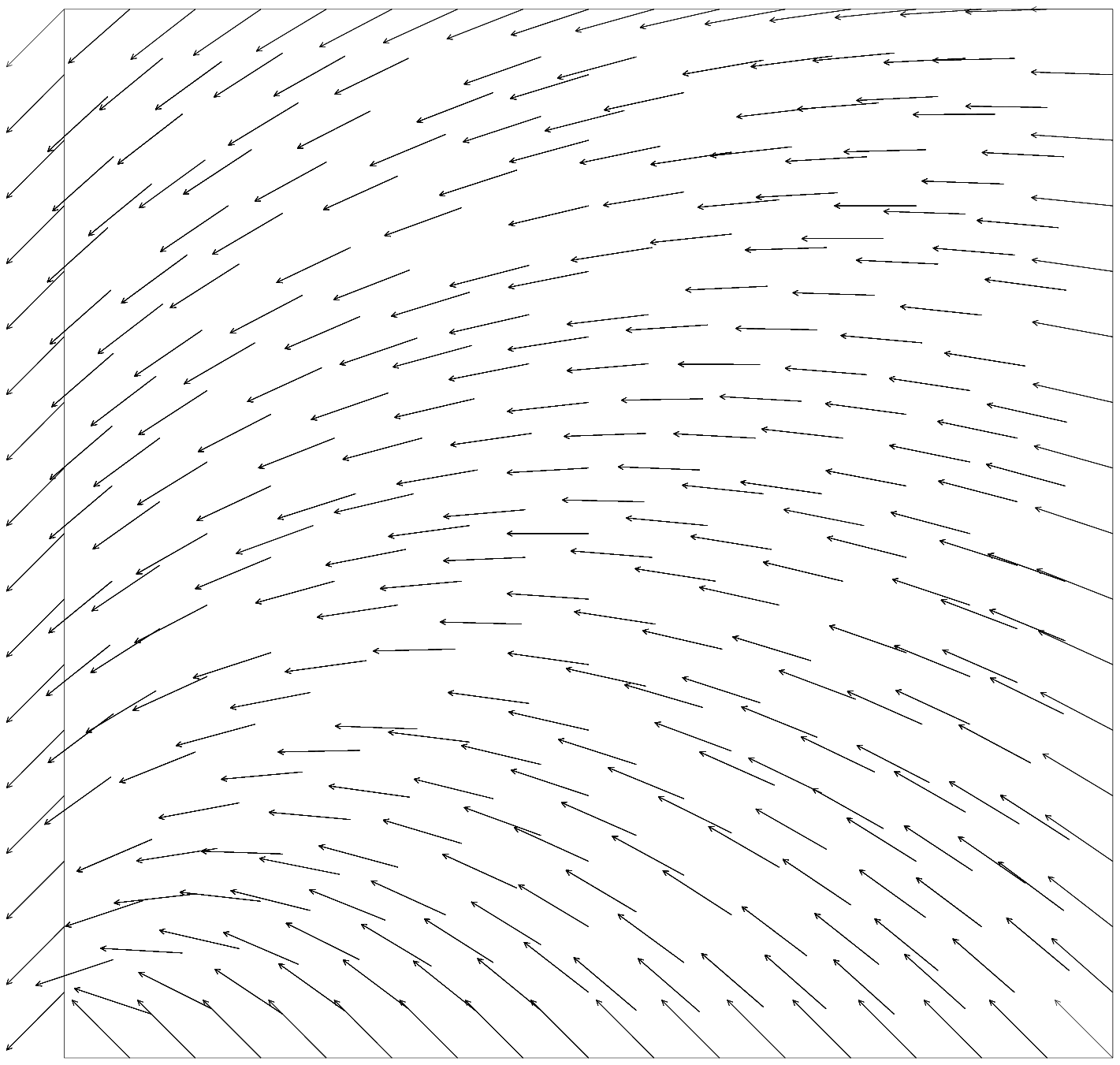}
\caption{Left: example of unstructured mesh, $N=5$. Right: plot of the
  velocity vector field.}\label{mesh_velocities}}
\end{figure}
\subsubsection{Dirichlet boundary conditions}
In table \ref{DirichletP1} we show the result of the computation when
Dirichlet boundary conditions are applied and piecewise affine
approximation is used on a sequence of unstructured meshes. We observe
that the solution exhibits
the pre-asymptotic convergence rate $h^{\frac32}$ under one refinement
before achieving the full second order convergence rate in $L^2$.

In table \ref{DirichletP2} similar data for second order polynomials
are presented. Here the asymptotic regime with full convergence is
obtained from the first refinement.
\begin{table}
\begin{center}
\begin{tabular}{|c|c|c|c|}
\hline
N & $\|u-u_h\|$ & $\|z_h\|$ & $|u_h|_{S_p} + |z_h|_{S_a}$ \\
\hline
3 &  0.038 (--)& 0.024 & 0.57\\ \hline
4 & 0.012 (1.7)& 0.0017 & 0.24\\ \hline
5 & 0.0024 (2.3)& 0.00043 & 0.11 \\ \hline
6 & 0.00043 (2.5)& 0.00012 & 0.052 \\ \hline
7 & 0.00010 (2.1)& 2.5E-05 & 0.025 \\ \hline
8 & 2.3E-05 (2.1)& 5.3E-06 & 0.012\\  \hline
\end{tabular}
\caption{Convergence orders of estimated quantities for the Dirichlet
  problem approximated using piecewise affine elements}\label{DirichletP1}
\end{center}
\end{table} 

\begin{table}
\begin{center}
\begin{tabular}{|c|c|c|c|}
\hline
N &$ \|u-u_h\|$ & $\|z_h\|$ &$ |u_h|_{S_p} + |z_h|_{S_a}$ \\
\hline
3 &  0.0014 (--)& 0.00041 & 0.024\\ \hline
4 &  0.00012 (3.5)& 4.6E-05 & 0.0044 \\ \hline
5 &  8.8E-06 (3.8)& 4.6Ee-06 & 0.00081 \\ \hline
6 &  8.0E-07 (3.5)& 6.6E-07 & 0.00017\\ \hline
7 &  8.3E-08 (3.3)& 8.2E-08 & 3.7E-05\\ \hline
\end{tabular}
\caption{Convergence orders of estimated quantities for the Dirichlet
  problem approximated using piecewise quadratic elements}\label{DirichletP2}
\end{center}
\end{table}

\subsubsection{Neumann boundary conditions}
We consider the same differential operator, but with (non-homogeneous)
Neumann-boundary conditions. This is exactly the problem considered in
\cite{CD11}. The average values of the approximate solutions have been
imposed using Lagrange multipliers. In tables \ref{NeumannP1} and
\ref{NeumannP2} we observe optimal convergence
rates once again as predicted by theory. Observe that in the case of
piecewise affine approximation the dual solution $z_h$ comes into the
asymptotic regime only on the finer meshes.
\begin{table}
\begin{center}
\begin{tabular}{|c|c|c|c|}
\hline
N & $\|u-u_h\|$ & $\|z_h\|$ & $|u_h|_{S_p} + |z_h|_{S_a} $\\
\hline
3 &  0.028 (--)& 0.028 (--)& 0.82\\ \hline
4 & 0.0066 (2.1)& 0.016 (0.8)& 0.32 \\ \hline
5 & 0.0016 (2.0)& 0.0058 (1.5)& 0.13\\ \hline
6 & 0.00039 (2.0)& 0.0015 (2.0)& 0.060\\ \hline
7 & 9.7E-05 (2.0)& 0.00031 (2.3) & 0.028\\ \hline
8 & 2.3E-05 (2.1)& 6.5E-05 (2.3)& 0.013\\  \hline
\end{tabular}
\caption{Convergence orders of estimated quantities for the Neumann
  problem approximated using piecewise affine elements}\label{NeumannP1}
\end{center}
\end{table}

\begin{table}
\begin{center}
\begin{tabular}{|c|c|c|c|}
\hline
N &$ \|u-u_h\|$ & $\|z_h\| $&$ |u_h|_{S_p} + |z_h|_{S_a}$ \\
\hline
3 &  0.00061 (--) & 0.0020 (--)& 0.030 \\ \hline
4 & 6.6E-05 (3.2) & 0.00040 (2.3)& 0.0054 \\ \hline
5 & 6.5E-06 (3.3)& 2.5E-05 (4.0)& 0.00099 \\ \hline
6 & 7.1E-07 (3.2)& 1.7E-06 (3.9) & 0.00020\\ \hline
7 & 7.9E-08 (3.2)& 1.4E-07 (3.6)& 4.2E-05\\ \hline
\end{tabular}
\caption{Convergence orders of estimated quantities for the Neumann
  problem approximated using piecewise quadratic elements}\label{NeumannP2}
\end{center}
\end{table}

\begin{table}
\begin{center}
\begin{tabular}{|c|c|c|c|c|}
\hline
N & $\|u-u_h\|$ & $\|z_h\|$ & $|u_h|_{S_p} + |z_h|_{S_a}$ & $\|\nabla (u -
u_h) \cdot n\|_{-\frac12,h,\partial \Omega}$\\
\hline
3 &   0.070 (--)& 0.59 (--)& 2.0 (--)& 2.7 (--)\\ \hline
4 &   0.074 (--)& 0.42 (0.49)& 0.79 (1.3) & 1.3 (1.1)\\ \hline
5 &   0.037 (1.0)& 0.30 (0.49) &0.30 (1.4)& 0.75 (0.80)\\ \hline
6 &   0.029 (0.35)& 0.26 (0.2)& 0.13 (1.2)& 0.51 (0.56)\\ \hline
7 &   0.024 (0.27)& 0.20 (0.37)& 0.054 (1.3)& 0.33 (0.62)\\ \hline
8 &   0.020 (0.26)& 0.16 (0.32)& 0.022 (1.3)& 0.21 (0.65)\\ \hline
\end{tabular}
\caption{Convergence orders of estimated quantities for the Poisson Cauchy
  problem approximated using piecewise affine elements}\label{P1cauchyPoissonfin}
\end{center}
\end{table}

\begin{table}
\begin{center}
\begin{tabular}{|c|c|c|c|c|}
\hline
N & $\|u-u_h\|$ & $\|z_h\|$ & $|u_h|_{S_p} + |z_h|_{S_a}$ & $\|\nabla (u -
u_h) \cdot n\|_{-\frac12,h,\partial \Omega}$\\
\hline
3 &  0.031 (--)& 0.062 (--)& 0.073 (--)& 0.92 (--)\\ \hline
4 &  0.022 (0.49)& 0.025 (1.3)& 0.014 (2.4)& 0.48 (0.94)\\ \hline
5 & 0.013 (0.76)& 0.014 (0.84)& 0.0025 (2.5)& 0.24 (1.0)\\  \hline
6 & 0.0088 (0.56)& 0.011 (0.35)& 0.00047 (2.4)& 0.13 (0.88)\\ \hline
7 & 0.0069 (0.35)& 0.0067 (0.72)& 8.8E-05 (2.8)& 0.080 (0.70)\\ \hline
\end{tabular}
\caption{Convergence orders of estimated quantities for the Poisson Cauchy
  problem approximated using piecewise quadratic elements,
  $\gamma=0.001$, $\gamma_{bc} = 10$}\label{cauchyPoissonP2fin}
\end{center}
\end{table}

%
\subsection{A Cauchy problem}
Since we have no complete theory for the ill-posed Cauchy problem we
will proceed with a more thorough numerical investigation. First we
consider the Cauchy problem obtained by taking $\kappa=0$ in \eqref{Cauchy_f}. Then we consider a Cauchy problem using the
convection--diffusion operator of \eqref{convdiff_forward} in two
different boundary configurations. For all test cases we use the
exact solution \eqref{exact_sol} and the stabilisation parameters
given above. We present the data for the quantities estimated in Lemma
\ref{stab_conv} and Theorem \ref{main}, but also the error in the
total diffusive flux in the discrete $H^{-1/2}(\partial \Omega)$ norm on the boundary.
\[
\|\nabla (u - u_h)\cdot n\|^2_{-\frac12,h,\partial \Omega}:=
\int_{\partial \Omega} h (\nabla (u - u_h)\cdot n)^2 ~\mbox{d}s.
\]
\subsubsection{Poisson's equation}
Here we consider the problem with $\kappa=0$ in \eqref{Cauchy_f}. We impose the Cauchy data, i.e. both Dirichlet and
Neumann data, on boundaries $x=0,\, 0<y<1$ and $y=1,\, 0<x<1$. In table \ref{P1cauchyPoissonfin} we
show the obtained errors when piecewise affine approximation is used
and in table \ref{cauchyPoissonP2fin} the results for piecewise quadratic
approximation. 

First note that in both cases one observes the optimal convergence of
the stabilisation terms predicted by Lemma \ref{stab_conv}. For the
$L^2$-norm of the error we observe experimental convergence orders
$h^\alpha$ with typically
$\alpha \sim 0.25$ for piecewise affine approximation and $\alpha
\sim 0.5$ for quadratic approximation.
Higher convergence
orders were obtained in both cases for the normal diffusive flux.
In figure \ref{penalty_plot}, we present a study of the $L^2$-norm
error under variation of the stabilisation parameter. 
The computations are made on one mesh, with $32$ elements
per side and the Cauchy problem is solved with $k=1,2$ and different
values for $\gamma_{F,1}=\gamma_{F,2}$ with $\gamma_{bc}=10$
fixed. The  level of $10 \%$ relative error is indicated by the
horizontal dotted line.
Observe that the robustness with
respect to stabilisation parameters is much better for quadratic
approximation. Indeed the $10\%$ error level is met for all parameter
values $\gamma_{i,F} \in [2.0E-5,1]$, whereas in the case of piecewise
affine approximation one has to take $\gamma_{1,F} \in [0.003,0.05]$ approximately.
Similar results for
the boundary penalty parameter not reported here showed that the
method was even more robust under perturbations of $\gamma_{bc}$.
In the left plot of figure \ref{Cauchy_error} we present the contour
plot of the error $u-u_h$ and in the right, the contour plot of $z_h$. In both cases
the error is concentrated on the boundary where no boundary conditions
are imposed.
\begin{figure}
\centering
\includegraphics[width=10cm]{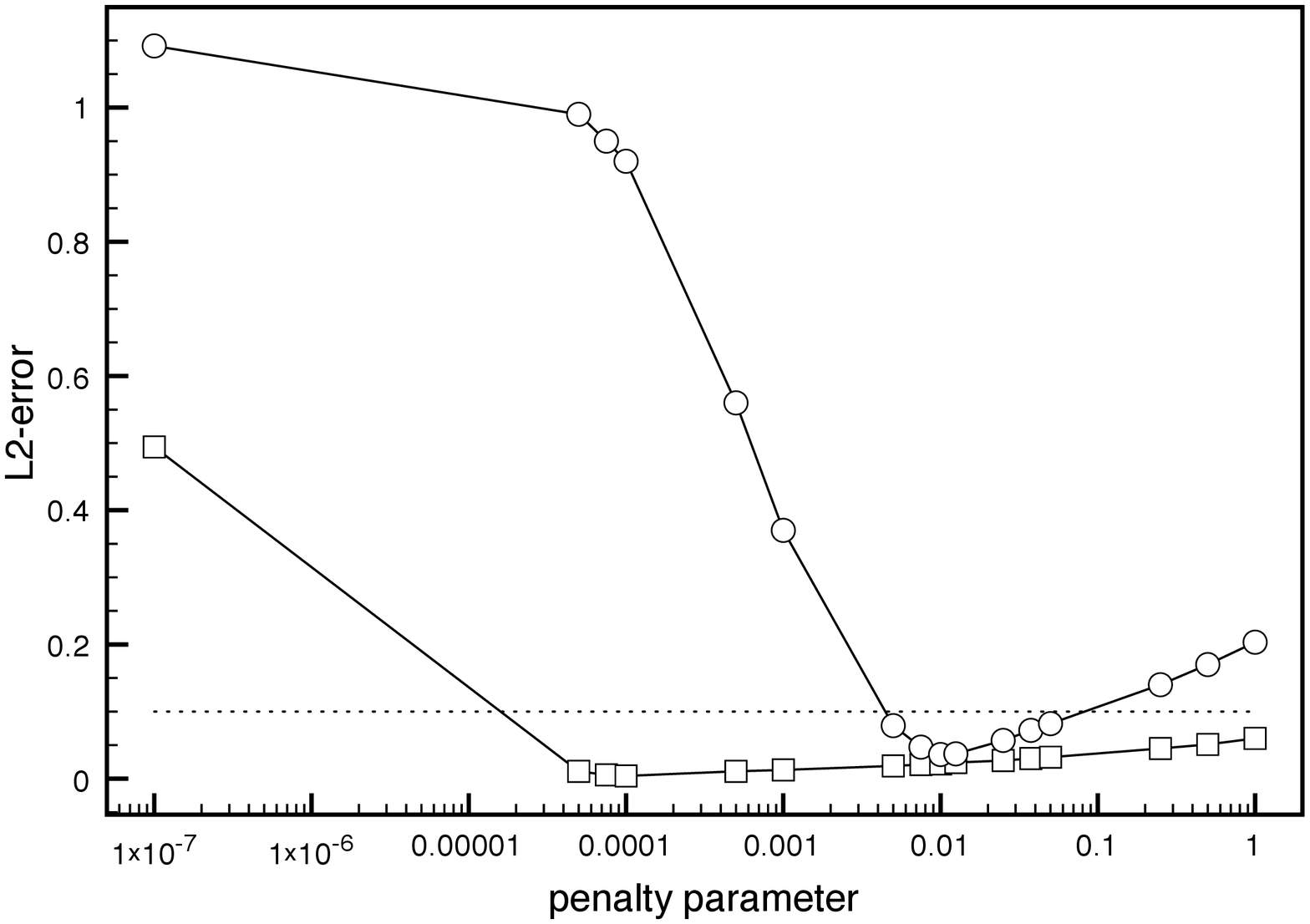}
\caption{Study of the $L^2$-norm
error under variation of the stabilisation parameter, circles: affine
elements, squares: quadratic elements}\label{penalty_plot}
\end{figure}
\begin{figure}
{\centering
\includegraphics[width=6.5cm]{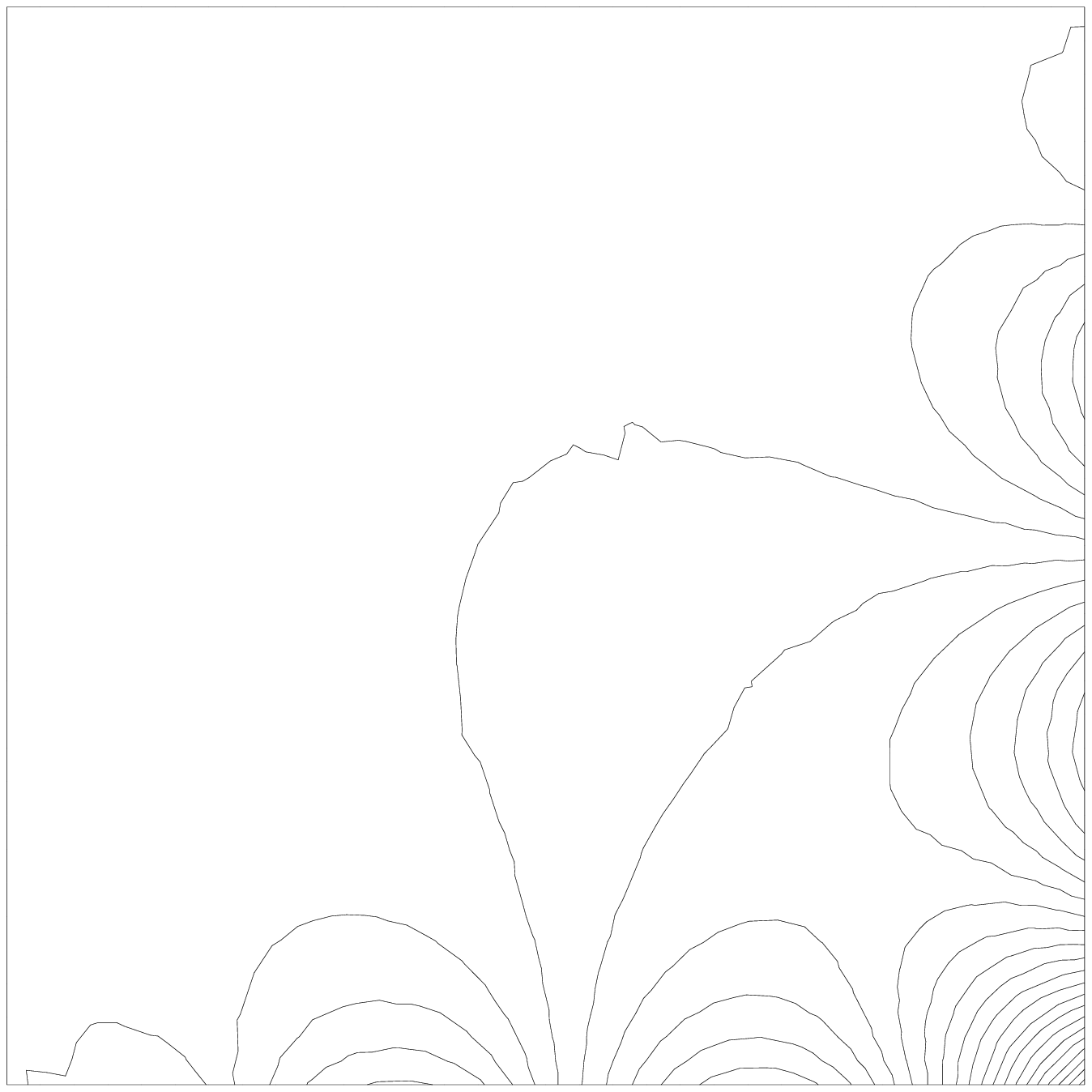}\hspace{-1cm}
\includegraphics[width=6.5cm]{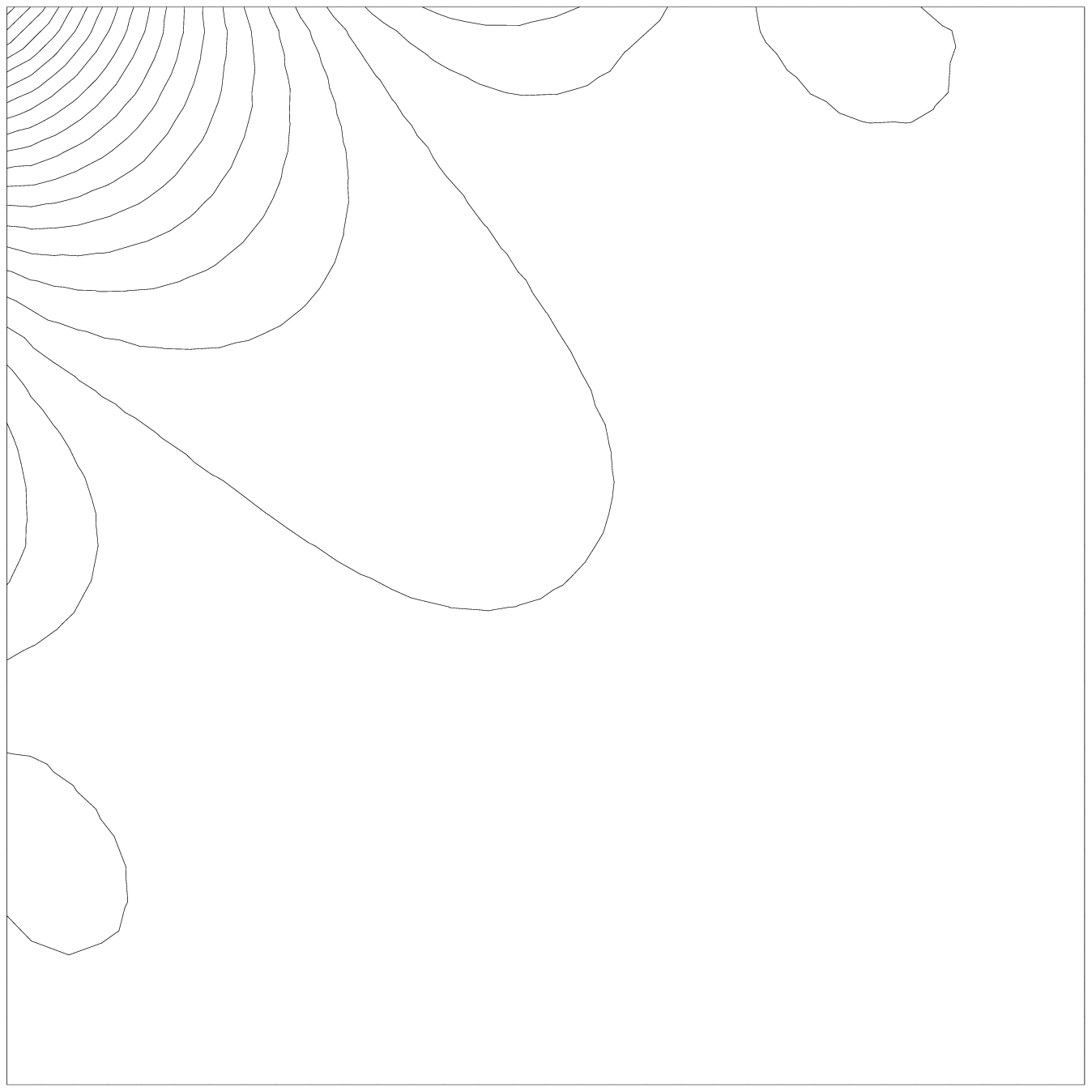}}
\caption{Contour plots of the error $u-u_h$ (left plot) and the
  error in the dual variable $z_h$ (right plot).}\label{Cauchy_error}
\end{figure}
\subsubsection{The noncoercive convection--diffusion equation}\label{noncoercive_exp}
As a last example we consider the Cauchy problem using the
noncoercive convection--diffusion operator \eqref{ellipt_oper}.
The stability of the problem depends strongly on where the boundary conditions
are imposed in relation to the inflow and outflow
boundaries. To illustrate this we propose two configurations. Recalling the right 
plot of figure \ref{mesh_velocities} we observe that the flow enters along the boundaries $y=0$, $y=1$ and
$x=1$ and
exits on the boundary $x=0$. Note that the strongest inflow takes
place on $y=0$ and $x=1$, the flow being close to parallel to the
boundary in the right half of the segment $y=1$. We propose the two
different Cauchy problem configurations:
\begin{itemize}
\item[Case 1.] We impose  Dirichlet and Neumann data on the two mixed
  boundaries $x=0$ and $y=1$.
\item[Case 2.] We impose Dirichlet and Neumann data on the two
  inflow boundaries $y=0$ and $x=1$. 
\end{itemize}
In the first case the outflow portion or the inflow portion of every
streamline are include in the Cauchy boundary whereas in the second
case the main part of the inflow boundary is included. This highlights
two different difficulties for Cauchy problems for the
convection--diffusion operator, in case 1 we must solve the problem
backward along the characteristics, essentially solving a backward
heat equation, whereas in case 2 the crosswind diffusion must
reconstruct missing boundary data.

In tables \ref{cauchyCDP1case1} - \ref{cauchyCDP2case2}, we report the results on the same sequence of unstructured meshes used
in the previous examples for piecewise affine and piecewise quadratic
approximations and the two problem configurations. First note that in
all cases the result of Lemma \ref{stab_conv} holds as
expected. Otherwise the method behaves very differently in the two
cases. For case 1 we observe better convergence orders than in the
case of the pure Poisson problem, typically $h^{\frac12}$ for affine
elements and $h$ for quadratic elements in the $L^2$-norm.
Even higher orders are obtained for the global diffusive flux in the
discrete $H^{-\frac12}$ norm. The dual variable $z_h$ on the other
hand has very poor convergence, although it is quite small on all
meshes in the case of quadratic approximation. Case 2 (control on main
part of the inflow) is clearly much more difficult. Convergence orders
for both the affine case and the quadratic case are poor (around $\sim h^{\frac15}$)
and uneven. The diffusive fluxes on the boundary nevertheless still converges approximately as
$h^{\frac12}$ in both cases. We conclude that the Cauchy
convection--diffusion problem is much less ill-posed if for each
streamline {\emph{either the inflow part or the outflow part lies in
    the controlled zone}}. The fact that we in case 2 controls more of
the inflow boundary is unimportant compared to the fact that both the
inflow and the outflow are unknown in the boundary portion around the
corner $(0,1)$.
\begin{table}
\begin{center}
\begin{tabular}{|c|c|c|c|c|}
\hline
N & $\|u-u_h\|$ & $\|z_h\|$ &$ |u_h|_{S_p} + |z_h|_{S_a}$ & $\|\nabla (u -
u_h) \cdot n\|_{-\frac12,h,\partial \Omega}$\\
\hline
3 &   0.032 (--)& 0.044  (--)& 1.6   (--)& 0.35  (--)\\ \hline
4 &   0.010 (1.7)& 0.020 (1.1)& 0.61 (1.4)& 0.13 (1.4)\\ \hline
5 &   0.0045 (1.2)& 0.034 (--)& 0.24 (1.3)& 0.048 (1.4)\\ \hline
6 &  0.0035 (0.36) & 0.052 (--)& 0.10 (1.3)& 0.018 (1.4)\\ \hline
7 &  0.0039 (--)& 0.056 (--)& 0.045 (1.2)& 0.0074 (1.3)\\ \hline
8 &   0.0026 (0.58)& 0.059 (--)& 0.020 (1.2)& 0.0031 (1.3)\\ \hline
\end{tabular}
\caption{Convergence orders of estimated quantities for the convection--diffusion Cauchy
  problem approximated using piecewise affine elements (case 1)}\label{cauchyCDP1case1}
\end{center}
\end{table}

\begin{table}\label{cauchyCDP1case2}
\begin{center}
\begin{tabular}{|c|c|c|c|c|}
\hline
N & $\|u-u_h\|$ & $\|z_h\|$ &$ |u_h|_{S_p} + |z_h|_{S_a}$ & $\|\nabla (u -
u_h) \cdot n\|_{-\frac12,h,\partial \Omega}$\\
\hline
3 &   0.13 (--)& 0.032 (--)& 1.74 (--)& 0.44 (--)\\ \hline
4 &   0.097 (0.42)& 0.012 (1.4)& 0.63 (1.5)& 0.23 (0.94)\\ \hline
5 &   0.075 (0.37)& 0.010 (0.26)& 0.24 (1.4)& 0.11 (1.1)\\ \hline
6 &   0.067 (0.16)& 0.010 (--)& 0.10 (1.3)& 0.070 (0.65)\\ \hline
7 &   0.063 (0.089)& 0.0097 (0.044)& 0.043 (1.2)& 0.047 (0.57)\\ \hline
8 &   0.056 (0.17)& 0.0082 (0.24)& 0.018 (1.3)& 0.030 (0.65)\\ \hline
\end{tabular}
\caption{Convergence orders of estimated quantities for the convection--diffusion Cauchy
  problem approximated using piecewise affine elements (case 2)}
\end{center}
\end{table}

\begin{table}\label{cauchyCDP2case1}
\begin{center}
\begin{tabular}{|c|c|c|c|c|}
\hline
N & $\|u-u_h\|$ & $\|z_h\|$ &$ |u_h|_{S_p} + |z_h|_{S_a}$ & $\|\nabla (u -
u_h) \cdot n\|_{-\frac12,h,\partial \Omega}$\\
\hline
3 &  0.0022 (--)& 0.0037 (--)& 0.096 (--)& 0.033 (--)\\ \hline
4 &  0.00054 (2.0)& 0.00089 (2.1)& 0.020 (2.3)& 0.0091 (1.9)\\ \hline
5 &  0.00024(1.2)& 0.0013 (--)& 0.0041 (2.3)& 0.0021 (2.1)\\ \hline
6 &  0.00012 (1.0)& 0.00078 (0.74)& 0.00096 (2.1)& 0.00047 (2.2)\\ \hline
7 & 5.6E-05 (1.1)& 0.00048 (0.70)& 0.00022 (2.1)& 0.00015 (1.6)\\ \hline
\end{tabular}
\caption{Convergence orders of estimated quantities for the convection--diffusion Cauchy
  problem approximated using piecewise quadratic elements (case 1)}
\end{center}
\end{table}

\begin{table}
\begin{center}
\begin{tabular}{|c|c|c|c|c|}
\hline
N & $\|u-u_h\|$ & $\|z_h\|$ &$ |u_h|_{S_p} + |z_h|_{S_a}$ & $\|\nabla (u -
u_h) \cdot n\|_{-\frac12,h,\partial \Omega}$\\
\hline
3 & 0.020 (--)& 0.0014 (--)& 0.074 (--)& 0.12 (--)\\ \hline
4 & 0.034 (--)& 0.00028 (2.3)& 0.013 (2.5)& 0.11 (0.12)\\ \hline
5 & 0.026 (0.39)& 0.00011 (1.4)& 0.0025 (2.4)& 0.065 (0.76)\\ \hline
6 & 0.024 (0.12)& 8.3E-05 (0.4)& 0.00046 (2.4)& 0.043 (0.60)\\ \hline
7 & 0.023 (0.06)& 3.6E-05 (1.2)& 8.7E-05 (2.4)& 0.029 (0.57)\\ \hline
\end{tabular}
\caption{Convergence orders of estimated quantities for the convection--diffusion Cauchy
  problem approximated using piecewise quadratic elements (case 2)}\label{cauchyCDP2case2}
\end{center}
\end{table}
\section{Concluding remarks}
We have proposed a framework for the design of stabilised finite
element methods for noncoercive and nonsymmetric
problems. The fundamental
idea is to use an optimisation framework to select the discrete
solution on each mesh. This also opens new venues for inverse problems
or boundary control problems, where Tichonov regularisation can been
introduced in the form of a stabilisation operator with optimal weak
consistency properties, eliminating the need to match a penalty
parameter and the mesh size to obtain optimal performance.
The method has some other
interesting features. In particular for piecewise affine approximation
spaces the discrete solution can be shown to exist under very mild
assumptions.
Both symmetric stabilisation methods and the Galerkin least squares
methods are considered in the analysis.
Convergence of the method is obtained formally under abstract
assumptions on the bilinear form that are shown to hold for three
nontrivial examples. The actual performance of the method in practice
depends crucially on the stability properties of the underlying PDE
and when these are unknown,
must be investigated numerically. Sometimes observed convergence
orders are unlikely to match those predicted in Theorem \ref{main},
(except possibly for very small $h$), due to
huge stability constants in the bound \eqref{smooth} (cf. the
Helmholtz equation for large wave numbers), ormore generally ill-posedness of the
dual problem (c.f. the Cauchy problem for Poisson's equation). Another problem
that may arise when ill-conditioned problems
are considered, is poor conditioning of the system matrix. Even in the case of piecewise
affine approximation the stabilisation corresponds to a very weak norm
and in case the underlying problem is ill-posed this must be expected
to show in the condition number. Clearly preconditioners for the linear systems
arising is an important open problem. Other subjects for future
work concerns the inclusion of hyperbolic problems in the framework
(see \cite{part2})
and the application of the method to data assimilation and boundary
control.
\section*{Acknowledgment} Partial funding for this research was
provided by EPSRC
(Award number EP/J002313/1).
\bibliographystyle{plain}   

\end{document}